\documentclass[11pt,leqno]{amsart}
\usepackage{amssymb}
\usepackage{xypic}
\begin{document}
\theoremstyle{plain}
\newtheorem{thm}{Theorem}[section]
\newtheorem*{thm1}{Theorem 1}
\newtheorem*{thm2}{Theorem 2}
\newtheorem{lemma}[thm]{Lemma}
\newtheorem{lem}[thm]{Lemma}
\newtheorem{cor}[thm]{Corollary}
\newtheorem{prop}[thm]{Proposition}
\newtheorem{propose}[thm]{Proposition}
\newtheorem{variant}[thm]{Variant}
\theoremstyle{definition}
\newtheorem{notations}[thm]{Notations}
\newtheorem{rem}[thm]{Remark}
\newtheorem{rmk}[thm]{Remark}
\newtheorem{rmks}[thm]{Remarks}
\newtheorem{defn}[thm]{Definition}
\newtheorem{ex}[thm]{Example}
\newtheorem{claim}[thm]{Claim}
\newtheorem{ass}[thm]{Assumption}
\numberwithin{equation}{section}
\newcounter{elno}                
\def\points{\list
{\hss\llap{\upshape{(\roman{elno})}}}{\usecounter{elno}}} 
\let\endpoints=\endlist


\catcode`\@=11
%
%
\def\opn#1#2{\def#1{\mathop{\kern0pt\fam0#2}\nolimits}} 
\def\bold#1{{\bf #1}}%
\def\underrightarrow{\mathpalette\underrightarrow@}
\def\underrightarrow@#1#2{\vtop{\ialign{$##$\cr
 \hfil#1#2\hfil\cr\noalign{\nointerlineskip}%
 #1{-}\mkern-6mu\cleaders\hbox{$#1\mkern-2mu{-}\mkern-2mu$}\hfill
 \mkern-6mu{\to}\cr}}}
\let\underarrow\underrightarrow
\def\underleftarrow{\mathpalette\underleftarrow@}
\def\underleftarrow@#1#2{\vtop{\ialign{$##$\cr
 \hfil#1#2\hfil\cr\noalign{\nointerlineskip}#1{\leftarrow}\mkern-6mu
 \cleaders\hbox{$#1\mkern-2mu{-}\mkern-2mu$}\hfill
 \mkern-6mu{-}\cr}}}
%
%

%
\def\:{\colon}
\let\oldtilde=\tilde
\def\tilde#1{\mathchoice{\widetilde{#1}}{\widetilde{#1}}%
{\indextil{#1}}{\oldtilde{#1}}}
\def\indextil#1{\lower2pt\hbox{$\textstyle{\oldtilde{\raise2pt%
\hbox{$\scriptstyle{#1}$}}}$}}
\def\pnt{{\raise1.1pt\hbox{$\textstyle.$}}}
%

%
\let\amp@rs@nd@\relax
\newdimen\ex@\ex@.2326ex
\newdimen\bigaw@
\newdimen\minaw@
\minaw@16.08739\ex@
\newdimen\minCDaw@
\minCDaw@2.5pc
\newif\ifCD@
\def\minCDarrowwidth#1{\minCDaw@#1}
\newenvironment{CD}{\@CD}{\@endCD}
\def\@CD{\def\A##1A##2A{\llap{$\vcenter{\hbox
 {$\scriptstyle##1$}}$}\Big\uparrow\rlap{$\vcenter{\hbox{%
$\scriptstyle##2$}}$}&&}%
\def\V##1V##2V{\llap{$\vcenter{\hbox
 {$\scriptstyle##1$}}$}\Big\downarrow\rlap{$\vcenter{\hbox{%
$\scriptstyle##2$}}$}&&}%
\def\={&\hskip.5em\mathrel
 {\vbox{\hrule width\minCDaw@\vskip3\ex@\hrule width
 \minCDaw@}}\hskip.5em&}%
\def\verteq{\Big\Vert&&}%
\def\noarr{&&}%
\def\vspace##1{\noalign{\vskip##1\relax}}\relax\let\amp@rs@nd@&\iffalse}\fi
 \CD@true\vcenter\bgroup\relax\let\\=\cr\iffalse}\fi\tabskip\z@skip\baselineskip20\ex@
 \lineskip3\ex@\lineskiplimit3\ex@\halign\bgroup
 &\hfill$\m@th##$\hfill\cr}
\def\@endCD{\cr\egroup\egroup}
%
\def\>#1>#2>{\amp@rs@nd@\setbox\z@\hbox{$\scriptstyle
 \;{#1}\;\;$}\setbox\@ne\hbox{$\scriptstyle\;{#2}\;\;$}\setbox\tw@
 \hbox{$#2$}\ifCD@
 \global\bigaw@\minCDaw@\else\global\bigaw@\minaw@\fi
 \ifdim\wd\z@>\bigaw@\global\bigaw@\wd\z@\fi
 \ifdim\wd\@ne>\bigaw@\global\bigaw@\wd\@ne\fi
 \ifCD@\hskip.5em\fi
 \ifdim\wd\tw@>\z@
 \mathrel{\mathop{\hbox to\bigaw@{\rightarrowfill}}\limits^{#1}_{#2}}\else
 \mathrel{\mathop{\hbox to\bigaw@{\rightarrowfill}}\limits^{#1}}\fi
 \ifCD@\hskip.5em\fi\amp@rs@nd@}
\def\<#1<#2<{\amp@rs@nd@\setbox\z@\hbox{$\scriptstyle
 \;\;{#1}\;$}\setbox\@ne\hbox{$\scriptstyle\;\;{#2}\;$}\setbox\tw@
 \hbox{$#2$}\ifCD@
 \global\bigaw@\minCDaw@\else\global\bigaw@\minaw@\fi
 \ifdim\wd\z@>\bigaw@\global\bigaw@\wd\z@\fi
 \ifdim\wd\@ne>\bigaw@\global\bigaw@\wd\@ne\fi
 \ifCD@\hskip.5em\fi
 \ifdim\wd\tw@>\z@
 \mathrel{\mathop{\hbox to\bigaw@{\leftarrowfill}}\limits^{#1}_{#2}}\else
 \mathrel{\mathop{\hbox to\bigaw@{\leftarrowfill}}\limits^{#1}}\fi
 \ifCD@\hskip.5em\fi\amp@rs@nd@}
%
%
\newenvironment{CDS}{\@CDS}{\@endCDS}
\def\@CDS{\def\A##1A##2A{\llap{$\vcenter{\hbox
 {$\scriptstyle##1$}}$}\Big\uparrow\rlap{$\vcenter{\hbox{%
$\scriptstyle##2$}}$}&}%
\def\V##1V##2V{\llap{$\vcenter{\hbox
 {$\scriptstyle##1$}}$}\Big\downarrow\rlap{$\vcenter{\hbox{%
$\scriptstyle##2$}}$}&}%
\def\={&\hskip.5em\mathrel
 {\vbox{\hrule width\minCDaw@\vskip3\ex@\hrule width
 \minCDaw@}}\hskip.5em&}
\def\verteq{\Big\Vert&}
\def\novarr{&}
\def\noharr{&&}
\def\SE##1E##2E{\slantedarrow(0,18)(4,-3){##1}{##2}&}
\def\SW##1W##2W{\slantedarrow(24,18)(-4,-3){##1}{##2}&}
\def\NE##1E##2E{\slantedarrow(0,0)(4,3){##1}{##2}&}
\def\NW##1W##2W{\slantedarrow(24,0)(-4,3){##1}{##2}&}
\def\slantedarrow(##1)(##2)##3##4{%
\thinlines\unitlength1pt\lower 6.5pt\hbox{\begin{picture}(24,18)%
\put(##1){\vector(##2){24}}%
\put(0,8){$\scriptstyle##3$}%
\put(20,8){$\scriptstyle##4$}%
\end{picture}}}
\def\vspace##1{\noalign{\vskip##1\relax}}\relax\let\amp@rs@nd@&\iffalse}\fi
 \CD@true\vcenter\bgroup\relax\let\\=\cr\iffalse}\fi\tabskip\z@skip\baselineskip20\ex@
 \lineskip3\ex@\lineskiplimit3\ex@\halign\bgroup
 &\hfill$\m@th##$\hfill\cr}
\def\@endCDS{\cr\egroup\egroup}
%
\newdimen\TriCDarrw@
\newif\ifTriV@
\newenvironment{TriCDV}{\@TriCDV}{\@endTriCD}
\newenvironment{TriCDA}{\@TriCDA}{\@endTriCD}
\def\@TriCDV{\TriV@true\def\TriCDpos@{6}\@TriCD}
\def\@TriCDA{\TriV@false\def\TriCDpos@{10}\@TriCD}
\def\@TriCD#1#2#3#4#5#6{%
\setbox0\hbox{$\ifTriV@#6\else#1\fi$}
\TriCDarrw@=\wd0 \advance\TriCDarrw@ 24pt
\advance\TriCDarrw@ -1em
\def\SE##1E##2E{\slantedarrow(0,18)(2,-3){##1}{##2}&}
\def\SW##1W##2W{\slantedarrow(12,18)(-2,-3){##1}{##2}&}
\def\NE##1E##2E{\slantedarrow(0,0)(2,3){##1}{##2}&}
\def\NW##1W##2W{\slantedarrow(12,0)(-2,3){##1}{##2}&}

\def\slantedarrow(##1)(##2)##3##4{\thinlines\unitlength1pt
\lower 6.5pt\hbox{\begin{picture}(12,18)%
\put(##1){\vector(##2){12}}%
\put(-4,\TriCDpos@){$\scriptstyle##3$}%
\put(12,\TriCDpos@){$\scriptstyle##4$}%
\end{picture}}}
\def\={\mathrel {\vbox{\hrule
   width\TriCDarrw@\vskip3\ex@\hrule width
   \TriCDarrw@}}}
\def\>##1>>{\setbox\z@\hbox{$\scriptstyle
 \;{##1}\;\;$}\global\bigaw@\TriCDarrw@
 \ifdim\wd\z@>\bigaw@\global\bigaw@\wd\z@\fi
 \hskip.5em
 \mathrel{\mathop{\hbox to \TriCDarrw@
{\rightarrowfill}}\limits^{##1}}
 \hskip.5em}
\def\<##1<<{\setbox\z@\hbox{$\scriptstyle
 \;{##1}\;\;$}\global\bigaw@\TriCDarrw@
 \ifdim\wd\z@>\bigaw@\global\bigaw@\wd\z@\fi
 \mathrel{\mathop{\hbox to\bigaw@{\leftarrowfill}}\limits^{##1}}
 }
 \CD@true\vcenter\bgroup\relax\let\\=\cr\iffalse}\fi
 \tabskip\z@skip\baselineskip20\ex@
 \lineskip3\ex@\lineskiplimit3\ex@
 \ifTriV@
 \halign\bgroup
 &\hfill$\m@th##$\hfill\cr
#1&\multispan3\hfill$#2$\hfill&#3\\
&#4&#5\\
&&#6\cr\egroup%
\else
 \halign\bgroup
 &\hfill$\m@th##$\hfill\cr
&&#1\\%
&#2&#3\\
#4&\multispan3\hfill$#5$\hfill&#6\cr\egroup
\fi}
\def\@endTriCD{\egroup} 
\newcommand{\mc}{\mathcal} 
\newcommand{\mb}{\mathbb} 
\newcommand{\surj}{\twoheadrightarrow} 
\newcommand{\inj}{\hookrightarrow} \newcommand{\zar}{{\rm zar}} 
\newcommand{\an}{{\rm an}} \newcommand{\red}{{\rm red}} 
\newcommand{\Rank}{{\rm rk}} \newcommand{\codim}{{\rm codim}} 
\newcommand{\rank}{{\rm rank}} \newcommand{\Ker}{{\rm Ker \ }} 
\newcommand{\Pic}{{\rm Pic}} \newcommand{\Div}{{\rm Div}} 
\newcommand{\Hom}{{\rm Hom}} \newcommand{\im}{{\rm im}} 
\newcommand{\Spec}{{\rm Spec \,}} \newcommand{\Sing}{{\rm Sing}} 
\newcommand{\sing}{{\rm sing}} \newcommand{\reg}{{\rm reg}} 
\newcommand{\Char}{{\rm char}} \newcommand{\Tr}{{\rm Tr}} 
\newcommand{\Gal}{{\rm Gal}} \newcommand{\Min}{{\rm Min \ }} 
\newcommand{\Max}{{\rm Max \ }} \newcommand{\Alb}{{\rm Alb}\,} 
\newcommand{\GL}{{\rm GL}\,} 
\newcommand{\ie}{{\it i.e.\/},\ } \newcommand{\niso}{\not\cong} 
\newcommand{\nin}{\not\in} 
\newcommand{\soplus}[1]{\stackrel{#1}{\oplus}} 
\newcommand{\by}[1]{\stackrel{#1}{\rightarrow}} 
\newcommand{\longby}[1]{\stackrel{#1}{\longrightarrow}} 
\newcommand{\vlongby}[1]{\stackrel{#1}{\mbox{\large{$\longrightarrow$}}}} 
\newcommand{\ldownarrow}{\mbox{\Large{\Large{$\downarrow$}}}} 
\newcommand{\lsearrow}{\mbox{\Large{$\searrow$}}} 
\renewcommand{\d}{\stackrel{\mbox{\scriptsize{$\bullet$}}}{}} 
\newcommand{\dlog}{{\rm dlog}\,} 
\newcommand{\longto}{\longrightarrow} 
\newcommand{\vlongto}{\mbox{{\Large{$\longto$}}}} 
\newcommand{\limdir}[1]{{\displaystyle{\mathop{\rm lim}_{\buildrel\longrightarrow\over{#1}}}}\,} 
\newcommand{\liminv}[1]{{\displaystyle{\mathop{\rm lim}_{\buildrel\longleftarrow\over{#1}}}}\,} 
\newcommand{\norm}[1]{\mbox{$\parallel{#1}\parallel$}} 
\newcommand{\boxtensor}{{\Box\kern-9.03pt\raise1.42pt\hbox{$\times$}}} 
\newcommand{\into}{\hookrightarrow} \newcommand{\image}{{\rm image}\,} 
\newcommand{\Lie}{{\rm Lie}\,} 
\newcommand{\CM}{\rm CM}
\newcommand{\sext}{\mbox{${\mathcal E}xt\,$}} 
\newcommand{\shom}{\mbox{${\mathcal H}om\,$}} 
\newcommand{\coker}{{\rm coker}\,} 
\newcommand{\sm}{{\rm sm}} 
\newcommand{\tensor}{\otimes} 
\renewcommand{\iff}{\mbox{ $\Longleftrightarrow$ }} 
\newcommand{\supp}{{\rm supp}\,} 
\newcommand{\ext}[1]{\stackrel{#1}{\wedge}} 
\newcommand{\onto}{\mbox{$\,\>>>\hspace{-.5cm}\to\hspace{.15cm}$}} 
\newcommand{\propsubset} {\mbox{$\textstyle{ 
\subseteq_{\kern-5pt\raise-1pt\hbox{\mbox{\tiny{$/$}}}}}$}} 
\newcommand{\sA}{{\mathcal A}} 
\newcommand{\sB}{{\mathcal B}} \newcommand{\sC}{{\mathcal C}} 
\newcommand{\sD}{{\mathcal D}} \newcommand{\sE}{{\mathcal E}} 
\newcommand{\sF}{{\mathcal F}} \newcommand{\sG}{{\mathcal G}} 
\newcommand{\sH}{{\mathcal H}} \newcommand{\sI}{{\mathcal I}} 
\newcommand{\sJ}{{\mathcal J}} \newcommand{\sK}{{\mathcal K}} 
\newcommand{\sL}{{\mathcal L}} \newcommand{\sM}{{\mathcal M}} 
\newcommand{\sN}{{\mathcal N}} \newcommand{\sO}{{\mathcal O}} 
\newcommand{\sP}{{\mathcal P}} \newcommand{\sQ}{{\mathcal Q}} 
\newcommand{\sR}{{\mathcal R}} \newcommand{\sS}{{\mathcal S}} 
\newcommand{\sT}{{\mathcal T}} \newcommand{\sU}{{\mathcal U}} 
\newcommand{\sV}{{\mathcal V}} \newcommand{\sW}{{\mathcal W}} 
\newcommand{\sX}{{\mathcal X}} \newcommand{\sY}{{\mathcal Y}} 
\newcommand{\sZ}{{\mathcal Z}} \newcommand{\ccL}{\sL} 
 \newcommand{\A}{{\mathbb A}} \newcommand{\B}{{\mathbb 
B}} \newcommand{\C}{{\mathbb C}} \newcommand{\D}{{\mathbb D}} 
\newcommand{\E}{{\mathbb E}} \newcommand{\F}{{\mathbb F}} 
\newcommand{\G}{{\mathbb G}} \newcommand{\HH}{{\mathbb H}} 
\newcommand{\I}{{\mathbb I}} \newcommand{\J}{{\mathbb J}} 
\newcommand{\M}{{\mathbb M}} \newcommand{\N}{{\mathbb N}} 
\renewcommand{\P}{{\mathbb P}} \newcommand{\Q}{{\mathbb Q}} 

\newcommand{\R}{{\mathbb R}} \newcommand{\T}{{\mathbb T}} 
\newcommand{\U}{{\mathbb U}} \newcommand{\V}{{\mathbb V}} 
\newcommand{\W}{{\mathbb W}} \newcommand{\X}{{\mathbb X}} 
\newcommand{\Y}{{\mathbb Y}} \newcommand{\Z}{{\mathbb Z}} 

\title{Hilbert-Kunz density function for  graded domains}

\author{Vijaylaxmi Trivedi and Kei-Ichi Watanabe}
\date{}
\address{School of Mathematics, Tata Institute of Fundamental Research, 
Homi Bhabha Road, Mumbai-40005, India }
\email{vija@math.tifr.res.in}
\address{Department of Mathematics, College of Humanities and Sciences, 
Nihon University, Setagaya-Ku, Tokyo 156-0045, Japan}
\email{watanabe@math.chs.nihon-u.ac.jp}
\thanks{}
\subjclass{}
\begin{abstract}We prove the existence of HK density function for a pair
$(R, I)$, where $R$ is a $\N$-graded  domain of finite type over a perfect 
field and $I\subset R$ is a graded ideal of finite colength. 
This generalizes our earlier result where one proves the existence of 
such a function for a pair $(R, I)$, where, in addition $R$ is standard 
graded.

As one of the consequences we show that if $G$ is a 
finite group scheme acting 
linearly on a polynomial ring $R$ of dimension $d$  then 
the HK density function $f_{R^G, {\bf m}_G}$, of the pair $(R^G, {\bf m}_G)$, is 
a piecewise polynomial function of degree $d-1$.

We also compute the HK density functions for $(R^G, {\bf m}_G)$, where 
$G\subset SL_2(k)$ is a finite group acting linearly on  the ring $k[X, Y]$.
\end{abstract}

\maketitle
\section{Introduction}

In this paper  a pair  $(R, I)$ is a {\em graded pair} if $R$ is an $\N$-graded domain of dimension $d\geq 2$ and finite type 
over  a perfect field $k$ of 
characteristic $p>0$, and $I$ is  a graded 
ideal of finite colength. The main result here is to prove the 
existence of the HK density function for such a pair.

The notion of HK density function was introduced in [T] for the purpose of 
studying the Hilbert-Kunz multiplicity (or HK multiplicity) $e_{HK}(R, I)$.
Recall that the  notion of HK multiplicity $e_{HK}(R, I)$  was 
introduced by P. Monsky [M] for an arbitrary  Noetherian ring $R$ 
(in characteristic $p>0$) and 
an ideal $I\subset R$ of finite colength. In the same paper he showed that it is  positive real number given by 
$$e_{HK}(R, I) = \lim_{n\to \infty}\frac{\ell(R/I^{[q]})}{q^d}.$$

The HK density function behaves well (when it exists)
for various operations like tensor products, Segre products etc.
Moreover  it is  a limit of a uniformly converging sequence 
(which could be suitably renormalized  to study a given specific property).

In [T], we proved  the existence of HK density function for a {\em standard 
graded pair} $(R,I)$, where by  a standard graded pair we mean 
  a graded pair, where $R$ is a standard graded ring (that is, $R$ is generated 
by $R_1$ as a $k$-algebra) in addition.

\vspace{10pt}

\noindent{\bf Theorem 1.1 [T]}.\quad {\em Let $(R, I)$ be a standard graded 
pair. Then for a 
finitely generated graded module $M$ over $R$  there is a sequence
$\{g_n(M_R,I):[0, \infty)\longto [0, \infty) \}_n$   
of compactly supported continuous and piecewise linear  functions 
 such that 

\begin{enumerate}
\item the sequence $\{g_n(M_R,I)\}_n$ is uniformly convergent.
Moreover
\item the HK density function 
 $f_{M_R, I}:[0, \infty]\longto [0, \infty)$ defined as 
$f_{M_R,I}(x) = \lim_{n\to \infty}g_n(M_R,I)(x)$ 
is a compactly supported continuous function, and 
$$e_{HK}(M,I) = \int_0^{\infty} f_{M_R,I}(x)dx.$$
\end{enumerate}}

Here, for a  finitely generated graded $R$-module $M$,  
$\{g_n(M_R,I):[0, \infty)\longto [0, \infty) \}_n$ denotes  the  sequence of 
functions 
 given as follows: 

For $x \geq 0 $, if
$x = (1-t)\frac{\lfloor xq \rfloor}{q} + (t)\frac{\lfloor xq +1 \rfloor}{q}$, 
for some $t\in [0, 1)$
then we define 
 $$g_n(M_R,I)(x) = \frac{1}{q^{d-1}}
\left((1-t) \ell(M/I^{[q]}M)_{\lfloor xq\rfloor} + 
(t)\ell(M/I^{[q]}M)_{\lfloor xq+1\rfloor}\right).$$

\vspace{10pt}

In this paper we generalize the above result to the case of {\em graded pair} 
$(R, I)$, where $R$ need not be standard graded.
(There are many interesting $\N$-graded rings which are not 
standard graded, for examples the ring of invariants and
the positive affine semigroup 
rings, in particular affine toric rings).

To do this  we need to
 generalize the notion of $g_n(M_R,I)$ (see Definition~\ref{d1})
 which 
 coincides with the above notion of $g_n(M_R, I)$ 
whenever  $\mbox{gcd}~\{n\mid R_n\neq 0\} =1$. 

\vspace{10pt}

More precisely we prove the following 

\vspace{10pt}

\begin{thm}\label{t1}{\em (Main Theorem)}.\quad If $M$ is  a finitely 
generated graded $R$-module, where 
$(R, I)$ is a graded pair then there is a 
  sequence 
$\{g_n(M_R,I):[0, \infty)\longto [0, \infty)\}_n$ of compactly supported 
continuous and piecewise linear  functions such that

\begin{enumerate}
\item $\{g_n(M_R,I)\}_{n\in \N}$ is a uniformly convergent sequence
of compactly supported  functions.
\item If $f_{M_R, I}:[0, \infty)\longto [0, \infty)$ given by  
$x \to  \lim_{n\to \infty} g_n(M_R,I)(x)$ then $f_{M_R,I}$ is a compactly
supported continuous function such that
$$(a)\quad f_{M_R,I} = (\rank~M)~f_{R,I}\quad\mbox{and}\quad (b)\quad 
e_{HK}(M, I) = \int_0^\infty f_{M_R,I}(x)dx. $$
\end{enumerate}
\end{thm}
\vspace{10pt}

We recall some key aspects of  the proof in the  situation of standard graded 
pair.

 If $R$ is a standard graded domain (which need not be normal) as in 
[T] with $I$ generated by homogeneous generators $f_1, \ldots, f_s$ 
of degrees $d_1, \ldots, d_s$
 then there exists a very ample invertible sheaf $\sO_X(D)$ on $X$ 
(associated  to a Cartier divisor $D$)  
 such that there is a graded inclusion 
$R \longto \oplus_{m\geq 0}H^0(X, \sO_X(mD))$ which is an isomorphism in all 
graded degrees  $m >> 0$. 
This gives us a short exact sequence of $\sO_X$-modules
\begin{equation}\label{de1} 
0\longto V\longto \oplus_i\sO_X((1-d_i)D) \longby{\phi} \sO_X(D)\longto 0,
\end{equation} 
$\phi(\sum_ia_i) = \sum_ia_if_i$.
Since $\sO_X(D)$ (in fact every $\sO_X(mD)$) is invertible the 
sequence (\ref{de1})
is locally split exact and hence taking its Frobenius pull backs 
($F:X\longto X$ is the Frobenius map induced 
by the map $\sO_X\longto \sO_X$ given by $x\to x^p$) 
gives the exact sequence  (here $q = p^n$)
\begin{equation}\label{de2}
0\longto  F^{n*}V\longto 
\oplus_i\sO_X((q-qd_i)D) \longby{\phi_{0,q}}
\sO_X(qD)\longto 0.\end{equation}
 
Since $\sO_X(mD)\tensor \sO_X(nD) \simeq \sO_X((m+n)D)$ 
tensoring (\ref{de2}) by $\sO_X(mD)$ we get the exact sequence 

\begin{equation}\label{de3}
0\longto  F^{n*}V\tensor\sO_X(mD)\longto 
\oplus_i\sO_X((m+q-qd_i)D) \longby{\phi_{m,q}}
\sO_X((m+q)D)\longto 0.\end{equation}
 
Now, for every $x\geq 0$, $\lfloor xq\rfloor = m+q$, for some integer $m$. Hence
we may define step functions 
\begin{equation}\label{de4}f_n(R,I)(x) :=  f_n(R,I)(\frac{m+q}{q})  =  
\frac{1}{q^{d-1}}\ell(R/I^{[q]})_{m+q} \end{equation}
$$=  
\frac{1}{q^{d-1}}\left[h^0(X, \sO_X(m+q)D)-
\oplus_ih^0(X, \sO_X(m+q-qd_i)D) + h^0(X, F^{n*}V\tensor\sO_X(mD))\right].$$
The sequence $g_n(R,I)$ is obtained from $f_n(R, I)$ in an obvious way.

In particular the computations depend on the cohomologies of the 
Frobenius pullbacks of the locally free sheaves $V$ and $\sO_X(D)$ and their 
twists (by the  line bundles $\sO_X(mD)$).

On the other hand if $R= \oplus_{m\geq 0} R_m$ is an arbitrary normal graded domain 
then by the 
theorem of Demazure (see Theorem~\ref{D} below), there is a $\Q$-divisor $D$  
such that $R_m = H^0(X, \sO_X(mD))$, for all $m$. But $\sO_X(D)$ 
need not be  invertible and the multiplication map  $\sO_X(mD)\tensor
\sO_X(nD) \longto \sO_X((m+n)D)$ need not be an isomorphism, in general. 
In particular the sequence  
(\ref{de1}) need not be locally split exact and $V$ may not be locally free. 
Hence a version of (\ref{de3}) 
cannot be derived from a single sequence like (\ref{de2}), 
and therefore ${\mbox{Ker}}~\phi_{m,q}$ does not come from  `twists of' a 
single sheaf
(unlike in the standard graded situation, where $ \mbox{Ker}~\phi_{m,q}= 
F^{n*}V\tensor \sO_X(mD)$, for all $m$ and $q$). 

However $\sO_X(mD)$, associated to such a $\Q$-divisor, does have some 
special properties which we exploit,  for example $\sO_X(mD)$ is 
 a reflexive sheaf of $\sO_X$-modules, hence invertible outside the 
singular locus of $X$.
As a result though one does not have a direct relation between 
the sequences (\ref{de3}) (as $m$ and $q$ vary), we are able to relate their
cohomologies by estimates 
$$ |h^0(X, \mbox{Ker}~\phi_{mp+n_1,qp})- p^{d-1}h^0(X,  \mbox{Ker}~\phi_{m,q})|
= O(m+q)^{d-2},\quad\mbox{for}\quad 0\leq n_1<p.$$

In particular the fact (Theorem~\ref{D}) 
that each $R_m$ is the space of 
sections  of the divisor  $mD$ allows us to give a simpler proof 
(than in [T]) for this more general setting (a graded pair).

However in [T] we prove the existence of the HK density function $f_{M_R,I}$ 
directly (and without the assumption that $R$ is a domain).
 Here we prove the Main Theorem when   $R$ is a domain, and the proof is 
in three steps: 
We prove the theorem when
$(M_R, I) = (R, I)$ and  where $R$
is a normal domain such that $\mbox{gcd}~\{m>0\mid R_m\neq 0\} = 1$. This 
is the main part. 
Then  we extend the result for the pair $(R,I)$, where 
 $R$ is a general graded domain. Then we further extend this to 
graded modules over such pairs.

We can extend the result to the case, when $R$ may not be
a domain, by 
defining 
$$f_{R, I} := \sum_{p\in \wedge}\lambda(M_P)f_{R/P, (I+P)/P},$$
where $\Lambda = \{p\in \Spec R\mid \dim R = \dim R/P\}$.
This is clearly an additive function and hence can be extended  
canonically to the notion of $f_{M_R, I}$.
In particular $\int f_{M_R, I}(x)dx = e_{HK}(M_R, I)$ as 
$e_{HK}(-)$ is an additive function. 

However, if $\mbox{gcd}\,\{m\mid R_m\neq 0\} = n_0$, say, 
the equality $$f_{R,I}(x) = \lim_{n\to \infty} {1/q^{d-1}}
\ell(R/I^{[q]})_{{\lfloor xq \rfloor}n_0}$$
may not hold any longer unless 
$\mbox{gcd}\,\{m\mid (R/P)_m\neq 0\} = n_0$, 
 for all $ P \in \Lambda$.

\vspace{10pt}

As a consequence of our Main Theorem (Theorem~\ref{t1}) we get the following

\begin{cor}\label{c3}Let  $(S, I)$ be  a graded pair of dimension $d>1$.
Suppose  there is  a graded ring $R$ with  a degree preserving 
map $S\subset R$ such that $R$ is $S$-finite, and 
 $\mbox{proj dim}_R (R/IR) < \infty$.

 Then  the HK density function
$f_{S, I}$  is a piecewise polynomial function  
of degree $d-1$, explicitly  given in terms of the graded Betti numbers of the 
resolution of $IR$.

In particular, if $R = k[X_1, \ldots, X_d]$ is a polynomial ring and 
 $G$ is a finite group (scheme) acting linearly on $R$ 
then 
for any graded pair $(R^G, I)$, where 
$R^G$ is  the ring of invariants, the function 
$f_{R^G,I}$ is a piecewise polynomial of degree $d-1$.
\end{cor}

We explicitly write down (in the tame case) the HK density function $f_{R^G, I}$, 
 where  $R= k[X_1, X_2]$ and $G\subset SL_2(k)$ a finite group and 
 $I$ is the graded 
maximal ideal of $R^G$.

\vspace{5pt} 

Similar to the case of standard graded pairs, the HK density function is 
multiplicative (Theorem~\ref{t2}) for the graded pairs too.  In particular 
the HK  density function of 
 the Segre product of two graded pairs can be written in terms of the 
HK density functions of those  pairs.

\section{preliminaries}
\begin{notations} By a {\em graded pair} $(R,I)$ we mean that
$R = \oplus_{m\geq 0}R_m$ is a Noetherian graded domain  of dimension 
$d\geq 2$, and of finite type over a  perfect field $k = R_0$  of 
characteristic $p>0$, and  
$I\subset R$ is a graded ideal such that $\ell(R/I)<\infty $.
\end{notations}

Let $(R,I)$ be a graded pair, and let $M$ be a finitely generated graded
 $R$-module. 
We extend the definition of $g_n(M_R,I)$ 
(given in the introduction)  as follows.  

\begin{defn}\label{d1} Let $n_0 = \mbox{gcd}~\{n\mid R_n\neq 0\}$. Then 
$f_n(M_R,I):[0, \infty)\longto [0, \infty)$
is the step function given by 
$$f_n(M_R,I)(x) = \frac{1}{q^{d-1}}\left(\ell(M/I^{[q]}M)_{\lfloor xq\rfloor n_0}+
\cdots +\ell(M/I^{[q]}M)_{\lfloor xq\rfloor n_0+n_0-1}\right).$$

If $x = (1-t)\frac{\lfloor xq \rfloor}{q} + (t)\frac{\lfloor xq +1 \rfloor}{q}$, 
for some $t\in [0, 1)$ then the function $g_n(M_R, I):[0. \infty)\longto [0, \infty)$ 
is given  by 
$$g_n(M_R,I)(x) = (1-t)f_n(M_R,I)(x) + (t)f_n(M_R, I)(x+\frac{1}{q}).$$

\end{defn}

In particular, each $g_n(R, I)$ is continuous, and the uniform convergence 
of the sequence $\{g_n(M_R,I)\}_n$ is equivalent 
to the uniform convergence of the sequence $\{f_n(M_R,I)\}_n$.

We also make the following observation that  the functions 
$g_n(M_R,I)$ and $f_n(M_R,I)$
are compactly supported with a bound on the support which is independent of $n$.

\begin{lemma}\label{l1}Each $g_n(M_R,I)$  is a compactly supported 
continuous function. 
Moreover, for a given pair $(M_R,I)$ there is a constant ${\tilde m}$ (independent of 
$n$) 
such that 
$$ \supp~g_n(M_R,I) \subseteq [0, {\tilde m}],\quad\mbox{for all}\quad n\geq 1.$$
In particular $\supp~f_n(M_R,I) \subseteq [0, {\tilde m}]$, for all $n \geq 1$.
\end{lemma}
\begin{proof} We choose the integers $s$, $l$, $m_{\mu}$ and $n_{\nu}$  as follows:
 Let $\mu(I) = s$. Let $J = \oplus_{m>0}R_m$ with a set of 
homogeneous generators  
$h_1, \ldots, h_{\mu}$  of degrees, say, 
$m_1\leq \cdots \leq m_{\mu}$ respectively. 
Let $l$ be an integer such that $J^l \subseteq I$.
Let $M$ be generated by 
homogeneous elements $g_1, \ldots, g_{\nu}$  of degrees 
$n_1\leq \cdots \leq n_{\nu}$.

Since $R_m = h_1R_{m-m_1}+\cdots + h_{\mu}R_{m-m_{\mu}}$ and 
$M_m = g_1R_{m-n_1}+\cdots + g_{\nu}R_{m-n_{\mu}}$,

$$m-n_\nu \geq (m_{\mu})lsq \implies  M_m \subseteq 
J^{ls q}M\subseteq I^{sq}M \subseteq I^{[q]}M.$$
Hence  $(M/I^{[q]}M)_{m} = 0$, for all $m\geq n_s+ (m_{\mu})lsq$.
\end{proof}

The following is a well known result.
\begin{lemma}\label{r2} Let $R = \oplus_{n\geq 0} R_n$ be a Noetherian graded 
domain such that $R_0$ is a field.
Then the following
three conditions are equivalent:
\begin{enumerate}
\item[(1)] $\mbox{gcd}~\{n >0\mid R_n\neq 0\} = n_0$. 
\item[(2)] $n_0 >0$ is the least integer with the property:
there is $m_1>0$ such that $R_{mn_0}\neq 
0$, for all $m\geq m_1$.
\item[(3)]  $n_0 >0$ is the least integer such that the quotient field of $R$
has an homogeneous element of degree $n_0$.
\end{enumerate}
\end{lemma}
\begin{proof}Left as an exercise for the reader.
\end{proof}

\section{The HK density functions for normal graded domains}

In this section we prove  the existence of the HK density 
function (in Proposition~\ref{p1}) for a graded pair $(R, I)$, where, 
in addition,  
 $R$ is a normal domain. We will make use of a technical lemma (Lemma~\ref{ML}), 
which we will prove  in Section~5.

For such a ring $R$ we will be  use the following  result of Demazure~[D].

\begin{thm}\label{D}{\em (Demazure)}.\quad Let $R = \oplus_{n\geq 0}R_n$ 
be a normal graded domain of 
finite type over a field $k$. Suppose there is an  homogeneous element $T$ of 
degree $1$ in the quotient field of $R$. Then for $X = \mbox{Proj}~R$, 
there exists a unique Weil $\Q$-divisor $D$ in $W\mbox{div}(X, \Q)$ 
such that $R_n = H^0(X, \sO_X(nD)).T^n$, for every $n\geq 0$.
\end{thm}

We recall some general facts about $\Q$-divisors.

\begin{notations}\label{nn1}Let $X$ be a normal projective  variety over a perfect 
field  $k$ (in our case $X= \mbox{Proj}~R$, where $R$ is a normal graded domain).

The set $W\mbox{div}(X)$ is the set of Weil divisors, where a Weil divisor is
a  formal 
sum of codimension $1$  integral subschemes (prime divisors) of $X$. 
The set  $$\mbox{Div}(X,\Q) = W\mbox{div}(X, \Q) =
W\mbox{div}(X)\tensor_{\Z}\Q,$$
 is the set of  formal linear combinations of codimension one 
integral subschemes of $X$ 
with  coefficients in $\Q$ (called $\Q$-divisors).
Let $K(X)$ denote the function field of $X$. For  
$D\in W\mbox{div}(X, \Q)$ the $\sO_X$-sheaf 
$\sO_X(D)$ is the sheaf  whose space of sections  on  
an open set $U\subset X$ is given by 
$$H^0(U, \sO_X(D)) = \{f\in K(X)\mid \mbox{div}(f)\mid_U+D\mid_U\geq 0\},$$
where $\mbox{div}(f) = \sum_iv_{D_i}(f)D_i$ and  
$v_{D_i}:K(X)\longto \Z\cup \{\infty\}$ is the discrete valuation of $K(X)$ 
corresponding to the prime divisor $D_i$.

In particular, if $D=\sum_ia_iD_i\in W\mbox{div}(X, \Q)$ is a formal sum of 
prime divisors $D_i$, where $a_i\in\Q$ then 
$\sO_X(D) = \sO_X(\lfloor D\rfloor)$, where 
$\lfloor D\rfloor =  \oplus_i \lfloor a_i\rfloor D_i$.
\end{notations}

For the following basic theory of reflexive sheaves we refer to [H1] 
(one can also look up the notes by [S] on his homepage).

\begin{defn} A coherent sheaf $\sF$ on $X$ is {\em reflexive} if the natural 
map of $\sO_X$-modules 
$\alpha:\sF \longto ({\sF}^{\wedge})^{\wedge}$ is an isomorphism, where 
$\sF^{\wedge} = \shom_{\sO_X}(\sF, \sO_X)$.

A rank one reflexive sheaf is invertible on the regular locus of $X$. In fact
$$\{\sO_X(D)\mid D\in W\mbox{div}(X)\} = 
 \{\mbox{rank}~~1~\mbox{reflexive subsheaves of}~~K(X)\}$$ 
and (even if $R$ is not normal)
$$\{\mbox{the Cartier divisors of}~~X\} = \{\mbox{invertible (hence reflexive) 
subsheaves  of}~~ K(X)\}.$$
Hence if $D$ is a Cartier divisor  then $D = \sum a_iD_i$, where $a_i\in \Z$ and hence 
$D = \lfloor D\rfloor$
\end{defn}

As we discussed earlier, 
in case $R$ is standard graded there is a Cartier divisor $D$ 
such that for $m>>0$,
 $R_m = H^0(X, \sO_X(mD))$. On the other hand if $R$ is graded normal domain 
then (by the above  theorem of 
   Demazure) there exists a $\Q$-divisor $D$ (which need not be Cartier, but 
some positive integer multiple of $D$ is a Cartier divisor) such that 
 $R_m = H^0(X, \sO_X(mD))$, for all $m\geq 0$.

\vspace{10pt}

We recall (in Lemma~\ref{r1}) some  relevant properties of $R$ and 
$\sO_X(nD)$ (see [D]).
 
By Lemma~\ref{r2} the existence of an  homogeneous element $T$ of degree $1$ 
in the quotient field of $R$ is 
equivalent to the condition that $R_m\neq 0$ for all $m>>0$ which is 
equivalent to saying that gcd $\{m>0 \mid R_m\neq 0\} =1$.

\begin{lemma}\label{r1} For $R$ and $D$ as in Theorem~\ref{D}, let
 $h_1, \ldots, h_\mu$ denote a set of 
homogeneous generators of $R$  as an $R_0$-algebra, of degrees
$m_1, \ldots, m_{\mu}$  respectively,
and  let $l_1 = \mbox{lcm}~(m_1, \ldots, m_\mu)$. 
Then 
\begin{enumerate}
\item[(a)] for  $n\in l_1\N$,   
the sheaf $\sO_X(nD)$ is a line bundle on $X$.
In particular the canonical multiplication map
$$\sO_X(nD)\tensor \sO_X(iD)\longto 
\sO_X((n+i)D)\quad\mbox{is an isomorphism, for all}\quad i.$$
\item[(b)] For $r = l_1\mu$ the line bundle $\sO_X(rD)$ is very ample on $X$.
\end{enumerate}
 \end{lemma}
\begin{proof}(a):
 The variety $X$ has the affine open cover 
$\{D_+(h_i)\}_i$, where 
$$\sO_X(nD)\mid_{D_+(h_i)} = \{f/h_i^{m} \mid\deg(f)-m\deg(h_i)= n, 
~f\in R_{\deg(f)}\}
= h_i^{n/m_i}\sO_X\mid_{D_+(h_i)}.$$
  is generated by the element
$h_i^{n/m_i}\in H^0(D_+(h_i), \sO_X\mid_{D_+(h_i)})$,  for all $i$.
 
Since $\sO_X(nD)$ is a Cartier divisor $\sO_X(nD+\lfloor iD\rfloor) = 
\sO_X(\lfloor (n+i)D\rfloor)$.
 
\vspace{5pt}

\noindent{(b)}:  If 
$R^{(r)} := \oplus_{m\geq 0}R_{rm}$ and $R^{(r)}_m:= R_{rm}$. Then  
$R^{(r)}$ is a standard graded ring, as for $m\geq 1$
$$R^{(r)}_m =  R_{mr} \subseteq R_{l_1}^{\mu(m-1)}R_{l_1\mu}
\subseteq R_{l_1\mu}^{m-1}R_{l_1\mu} = (R_1^{(r)})^m.$$
Since  $X = \mbox{Proj}~R^{(r)}$, 
 the sections of $\sO_X(rD)$  give a closed immersion of 
$X$ into $\P_k^h$ where $h = h^0(X,\sO_X(rD))-1$.
\end{proof}

\vspace{5pt}

In the rest of the section we  have the following notations.

\begin{notations}\label{n5}The pair $(R,I)$ is a  fixed  graded pair,
 where $R$ is a normal graded domain  and 
gcd~~$\{n\mid R_n\neq 0\} = 1$.

We fix a homogeneous element $T$ of degree $1$ in the quotient field of $R$.

Let $D \in Div(X, \Q)$ be the divisor,  as in  Theorem~\ref{D}
so that $R = R(X, D) = \oplus_{n\geq 0}H^0(X, \sO_X(nD)).T^n$.

We fix
$r\in\N$, so that $\sO_X(rD)$ is a 
very ample divisor on $X$ and  $R_m\neq 0$, for all $m\geq r$.

For $I$, we fix 
a set of homogeneous generators $f_1, \cdots, f_{s}$ 
 of degrees $d_1, \ldots, d_s$
respectively.

\vspace{5pt}
For the sake of abbreviation we adopt the following Notations.
\vspace{5pt}

Let $\sO_n = \sO_X(nD)$.

\vspace{5pt}
Let $\sL = \sO_r$ be the very ample line bundle on $X$. 

\vspace{5pt}
Let $m_{q, d_i} = m+q-qd_i$ 
where $q=p^n$, for some $n\geq 1$.

\vspace{5pt}
Let $(m_{q, d_i})  = {\lfloor \frac{m+q}{r}\rfloor}r -qd_i$.

\end{notations}
\vspace{5pt}

Since gcd $\{m>0\mid R_m \neq 0\} =1$, the definition of 
 the sequences $\{f_n(R, I)\}_n$ and $\{g_n(R, I)\}_n$  
is  same as  in the case of  standard graded pair (given in [T]).

\begin{defn} For the pair $(R, I)$,  the function   
  $\{f_n(R,I):[0, \infty)\to [0, \infty)\}_{n\in \N}$ is given by 
$$f_{n}(R,I)(x) = {1}/{q^{d-1}}\ell({R}/{I^{[q]}})_{\lfloor xq\rfloor},
~~\mbox{where}~~q=p^n$$ 
and, for
 $x= (1-t)m/q + t(m+1)/q$, where $t\in [0, 1)$,
 $$g_n(R,I)(x) = (1-t)f_n(R,I)({m}/{q}) + (t)f_n(R,I) 
({m+1}/{q}).$$
\end{defn}

For given $m\in N$ and $q=p^n$,
we consider  the following short exact sequence of $\sO_X$-modules

\begin{equation}\label{e10}0\longto \sF_{m,q}\longto  \oplus_{i=1}^s \sO_{m_{q, d_i}}
\longby{\varphi_{m, q}} \sO_{m+q}
\longto 0,\end{equation}
where $\varphi_{m,q}(a_1,\ldots, a_s) = \sum_ia_if_i^q$.

Then
$$\begin{array}{lll}
f_n(R,I)(\frac{m+q}{q}) & = & 
 \frac{1}{q^{d-1}}\ell(R/I^{[q]})_{m+q} = 
 \frac{1}{q^{d-1}}\left[\ell(R_{m+q})-\sum_{i=1}^s\ell(f_i^qR_{m+q-qd_i})\right]\\\\
& = &  \frac{1}{q^{d-1}}\left[h^0(X, \sO_{m+q})-
\oplus_ih^0(X, \sO_{m_{q,d_i}}) + h^0(X, \sF_{m,q})\right].\end{array}$$

 To compare $f_n$ and $f_{n+1}$, we use the following crucial technical  result, 
which will be proved below in Section~5.

\begin{lemma}\label{ML}{\em (Main Lemma)}.\quad For  $\sF_{m,q}$ as in (\ref{e10}), 
there exists a  constant $C$ such that, 
for all $m\geq 0$ and $q = p^n$ and $0\leq n_1<p$,
$$|h^0(X, \sF_{mp+n_1, qp})-p^{d-1}h^0(X, \sF_{m,q})| \leq  C (mp+qp)^{d-2},$$
and 
$$|h^0(X, \sO_{m_qp+n_1})-p^{d-1}h^0(X, \sO_{m_q})|\leq  C (mp+qp)^{d-2},$$
where  $m_q = m_{q,d_i}$,  $m_q = (m_{q,d_i})$, for 
$1\leq i \leq s$, or $m_q =m+q$.\end{lemma}

\vspace{10pt}
The following proposition proves the existence of the HK density function for 
normal graded domains.

\begin{propose}\label{p1} If $R = \oplus_{n\geq 0}R_n$ is a normal graded domain
 and gcd~~$\{n\mid R_n\neq 0\} = 1$ then for a graded pair $(R, I)$ 
the sequence $\{f_n(R,I)\}_n$  is uniformly convergent.\end{propose} 
\begin{proof} For brevity, in the rest of the  
proof, we denote $f_n(R,I)$ by $f_n$.

Let $x\geq 1$. For $q=p^n$ and $m+q\leq xq <m+q+1$,
$$f_{n}(x) = {1}/{q^{d-1}}\ell(R/I^{[q]})_{\lfloor xq\rfloor} = 
{1}/{q^{d-1}}\ell(R/I^{[q]})_{m+q}.$$

Therefore there is $n_1$ such that $0\leq n_1 < p$ and   
$$f_{n+1}(x) = {1}/{(qp)^{d-1}}\ell(R/I^{[qp]})_{mp+qp+n_1}.$$

For $\sO_m = \sO_X(mD)$ consider
the short exact sequences of $\sO_X$-modules (as in (\ref{e10})

$$0\longto \sF_{m,q}\longto  \oplus_{i=1}^s \sO_{m+q-qd_i}
\longby{\varphi_{m, q}} \sO_{m+q}
\longto 0,$$

$$0\longto \sF_{mp+n_1,qp}\longto  \oplus_{i=1}^s \sO_{mp+qp+n_1-qpd_i}
\longby{\varphi_{mp+n_1, qp}} \sO_{mp+qp+n_1}\longto 0.$$
Therefore 
$$f_n(x) = \frac{p^{d-1}[h^0(X, \sO_{m+q})-
\sum_ih^0(X, \sO_{m+q-qd_i}) + h^0(X, \sF_{m,q})]}{(qp)^{d-1}},$$

$$f_{n+1}(x) = \frac{[h^0(X, \sO_{mp+qp+n_1})-
\sum_ih^0(X, \sO_{mp+qp+n_1-qpd_i}) + h^0(X, \sF_{mp+n_1,qp})]}{(qp)^{d-1}}.$$

By the Main Lemma~\ref{ML},
 there is a constant $C_0>0$ such that 
$$|f_n(x) -f_{n+1}(x)| \leq C_0 {(mp+qp)^{d-2}}/{(qp)^{d-1}}.$$

Since $\supp~f_n \subseteq [0, {\tilde m}]$, where ${\tilde m}$ is as in Lemma~\ref{l1},
 we can assume $(m+q)/q \leq {\tilde m}$ and hence there is  a constant 
$C_1$ 
 such that 
$$|f_n(x)-f_{n+1}(x)| \leq  
 {C_1}/{qp}, \quad\mbox{for all}\quad x\geq 1.$$

We can further  choose  $C_1$ such that the above inequality also holds for all
$0\leq x\leq 1$:
Because 
if $0\leq x<1$ then 
$$f_n(x) = {\ell(R)_{\lfloor xq\rfloor}}/{q^{d-1}} = 
{P({\lfloor xq\rfloor})}/{q^{d-1}}
~~\mbox{for all}\quad q = p^n >>0,$$
 where $P(x)\in \Q[X]$ is  the  Hilbert polynomial of $R$ hence of degree $d-1$.
This proves the proposition.
\end{proof}

\section{The Main theorem}
In this section we will prove that  the Main Theorem holds for the general 
graded pairs $(R,I)$. (Here we still assume the Main Lemma~\ref{ML} which 
will be proved in the next section).

Throughout this section 
$(R, I)$ is a graded pair and 
gcd $\{m\mid R_m\neq 0\} = n_0$.

 For a finitely generated graded $R$-module, 
the functions $f_n(M_R,I)$ and $g_n(M_R,I)$ are as in Definition~\ref{d1}.

\begin{rmk}\label{n6}If $S\longto R$ is a  degree preserving module-finite 
map of graded domains, where $(S, I)$ is a graded pair and  
$n_0 = \mbox{gcd}~\{n\mid R_n\neq 0\}$ and 
$m_0 = \mbox{gcd}~\{n\mid S_n\neq 0\}$,
then the HK density function of $(R, IR)$ as a module over itself is  ($q=p^n$)
  $$f_{R, IR}(x) : = \lim_{n\to \infty} f_n(R, IR)(x) 
= \frac{1}{q^{d-1}}\left(\ell(R/I^{[q]}R)_{\lfloor xq\rfloor n_0}\right).$$

Whereas  the HK density function of $(R, IR)$ as a module over $S$ is 
$$f_{R_S, I}(x) := \lim_{n\to \infty} f_n(R_S, I)(x)
 = \frac{1}{q^{d-1}}\left(\ell(R/I^{[q]}R)_{\lfloor 
xq\rfloor m_0}
+
\cdots +\ell(R/I^{[q]}R)_{\lfloor xq\rfloor m_0+m_0-1}\right).$$
The existence of both the limits is shown in the  following Theorem~\ref{t1}.
\end{rmk}

We use the following lemma to reduce the problem of convergence
of $\{f_n(M_R, I)\}$ to the problem of convegence of $\{f_n(S, IS)\}$, 
where $S$ is the normalization of  $R$ in $Q(R)$.

\begin{lemma}\label{l12} If \mbox{gcd}~$\{m>0\mid R_m\neq 0\} = 1$ 
and $N$, $N'$ are  finitely generated graded $R$-modules with 
 the  exact sequence of graded $R$-linear maps 
\begin{equation}\label{se}
0 \longto  N\longby{\phi} N'\longto Q''\longto 0 \end{equation}
 such that 
the {\em supp dim} $\,Q''< d$ and the map $\phi$ is of degree $0$.

\vspace{5pt}

Then  the sequence $\{f_n(N_R,I)\}_n$ is uniformly convergent if and only if 
$\{f_n(N_R',I)\}_n$ is so. Moreover in that case 
$$\lim_{n\to \infty}f_n(N_R,I) =  \lim_{n\to \infty}f_n(N_R',I).$$
\end{lemma}
\begin{proof} Here if $M$ is a graded $R$-module then 
the function
 $f_n(M_R, I):[0, \infty)\longto [0, \infty)$ is given by 
$x\to  {\ell(M/I^{[q]}M)_{\lfloor xq\rfloor}}/{q^{d-1}}$, where
$q=p^n$.

Let $I$ have homogeneous generators $f_1, \ldots, f_s$ of degree 
$d_1, \ldots, d_s$ respectively. Then for 
 any graded $R$-module $M$, we define 
$$\Phi_M:\oplus_i^sM(-qd_i) \longto M\quad\mbox{given by}\quad
(m_1, \ldots, m_s) \to \sum_if_i^qm_i$$ 

This gives graded degree $0$ maps of graded $R$-modules, functorial in $M$,  
$$0\longto \mbox{Ker}~\Phi_M \longto \oplus_{1}^sM(-qd_i) \longby{\Phi_M} M\quad
\longto \mbox{Coker}~\Phi_M\longto 0.$$

Now the snake lemma applied to (\ref{se}) gives the following exact sequence 
of graded $R$-modules

$$ \longto
\mbox{Ker}~\Phi_{Q''} \longto \mbox{Coker}~\Phi_{N} \longto \mbox{Coker}~\Phi_{N'} 
\longto \mbox{Coker}~\Phi_{Q''} \longto 0,$$
where $f_n(N'_R, I)(\frac{m+q}{q}) = \ell(\mbox{Coker}~\Phi_{N'})_{m+q}$ and
$f_n(N_R, I)(\frac{m+q}{q}) = \ell(\mbox{Coker}~\Phi_{N})_{m+q}$.

Let $C_{Q''}$ be constant such that, for all $m>0$, 
$\ell(Q''_m)\leq C_{Q''}m^{d-2}$ (such a
constant exists by the hypothesis on support dimensions). 

$$|\ell(\mbox{Coker}~\Phi_{N'})_{m+q} -\ell(\mbox{Coker}~\Phi_{N})_{m+q}|
\leq 2C_{Q''}(m+q)^{d-2}.$$

Now, for  $x\geq 1$ we have   $m+q\leq xq<m+q+1$ for some $m\geq 0$ and 
so we have  
$$|{f_n}(N_R,I)(x)-{f_n}(N_R',I)(x)|\leq 2C_{Q''}x_0^{d-2}/q,$$
  
where by Lemma~\ref{l1},  we may fix an  $x_0$ such that 
 $\supp {f_n}(N_R,I)$ and $\supp {f_n}(N_R',I)$ are  subsets of 
$[0, x_0]$, for all $n\geq 1$.

If $0\leq x <1$ then $m\leq xq <m+1$, for some $m < q$.
It is easy to check that in this case  
$$|{f_n}(N_R,I)(x)-{f_n}(N_R',I)(x)|= 2C_{Q''}m^{d-2}/{q}^{d-1} \leq
2C_{Q''}/{q}.$$
This proves the lemma.
\end{proof}

Now we are ready to prove the Main Theorem.

\vspace{10pt}

\noindent{\underline {Proof of the Main Theorem~\ref{t1}}}:\quad Let $\mbox{gcd}~\{n\mid R_n \neq 0\} = n_0$. 
Let $S = R^{(n_0)}$, where the $n^{th}$ degree component of $R^{(n_0)}$ 
is $R_{nn_0}$. 
Then $S$ is a graded domain, where  
{gcd}~$\{n\mid S_n \neq 0\} = 1$. (Note that $S=R$ as rings, but the grading is changed.)

Note that assertion (2)~(b) follows from assertion (1).

It is easy to prove that the  
 uniform convergence of $\{g_n(M_R,I)\}_n$ is equivalent to  
the uniform convergence of $\{f_n(M_R,I)\}_n$.

We first prove the theorem for $M=R$, where 
 it is sufficient  to prove  the uniform convergence of  
$\{f_n(R,I)\}_n$. 
By definition,  
 $f_n(R,I) = f_n(S, I)$, for all $n$.

Let ${\tilde S} = \oplus_{n}{\tilde S}_n$ denote the normalization of $S$ in 
its quotient field then 
 the inclusion map $S\longto {\tilde S}$ is a module finite graded map of 
degree $0$,  
and  we have the short exact sequence 
of graded $S$-modules
$$0\longto S\longto {\tilde S} \longto Q''\longto 0,$$
where support dim~$Q'' \leq d-1$.

 By  Proposition~\ref{p1},
the sequence $\{f_n({\bar S},I{\bar S})$  is uniformly convergent.
But  $f_n({\bar S}_S, I) = f_n({\bar S}, I{\bar S})$. Hence the uniform 
convergence of $\{f_n(R,I)\}_n$   follows by Lemma~\ref{l12}.

\vspace{10pt}
We now consider the general case of a finite graded module $M$.

\vspace{10pt}
Let ${\bar M} = \oplus_n{\bar M}_n$, where 
${\bar M}_n = M_{nn_0}+\cdots +M_{nn_0+n_0-1}$ denotes the degree $n$ 
component of ${\bar M}$.
If $M$ is generated by  homogeneous elements  $g_1, \ldots, g_{\nu}$ as an 
$R$-module then 
${\overline{M}}$ is generated by $g_1, \ldots, g_{\nu}$ as an 
$S$-module.
Hence ${\overline{M}}$ is a finitely generated graded $S$-module. 
Also, for all $n\geq 1$,  $f_n(M_R,I) = f_n({\overline{M}}_S, IS)$
and $\rank_RM = \rank_S{\overline{M}}$.

\vspace{5pt}

\noindent{\bf Claim}.\quad There exists 
 an exact sequence of graded $S$-modules
$$0\longto  \oplus^{n_1}S(-a) \longby{\phi} {\overline{M}}\longto Q''\longto 0,$$
where $\phi$ is a graded map of degree $0$ and $\dim~(Q'')\leq d-1$.

\vspace{5pt}

\noindent{{\underline{Proof of the claim}}}:\quad For  the multiplicatively 
closed set
$T = S\setminus \{0\}$, the  $T^{-1}S$-module
$T^{-1}{\overline {M}}$ is  free of finite rank, say $n_1$ and is 
 generated by a finite set of homogeneous elements. Hence  
 we can choose homogeneous elements $m_1,\ldots, m_{n_1}$ in 
${\overline{M}}$ of degrees
$d_1,\ldots, d_{n_1}$ respectively such that the
$m_i's$ give a basis for $T^{-1}{\overline{M}}$.

Since gcd $\{n\mid S_n\neq 0\} =1$, we have $m_0$ such that $S_m\neq 0$, for all
$m\geq m_0$. 
Let $a>0 $ such that $a\geq \max\{m_0+d_i, m_0\}_i$ and 
let  $s_i\in S_{a-d_i}\setminus \{0\}$. Then 
$s_1m_1, \ldots, s_{n_1}m_{n_1}\in {\overline{M}}$ are homogeneous elements 
(each of degree  $a$) 
 and generate $T^{-1}{\overline {M}}$ as $T^{-1}S$-module.
Hence we have a generically isomorphic map 
$\oplus^{n_1}S(-a) \longto  {\overline {M}}$ of graded $S$-modules of degree $0$.
The map is injective as $S$ is  a domain.
This proves the claim.
\vspace{5pt}

Now the theorem (1) and (2)~(a) follows from  
Lemma~\ref{l12}.\hspace{50pt}\hfill $\Box$

Note that assertion (2)~(b) follows from assertion (1).

\vspace{10pt}

As we remarked earlier (Remark~\ref{n6}), 
for a finite map  $S\longto R$ as in Notations~\ref{n6}, 
 the two HK density functions, for the pair $(R, IR)$,
namely  $f_{R,IR}$ and   $f_{R_S, I}$ 
need not be the same functions
 but can be recovered from each other as follows.

\vspace{10pt}

\begin{lemma}\label{l13}Let $S\longrightarrow R$ be the module-finite map as in 
Notations~\ref{n6}. Let $m_0 = \mbox{gcd}\{n >0\mid S_n\neq 0\}$ and $n_0 = 
\mbox{gcd}\{n>0\mid R_n\neq 0\}$ then 
$$(l_0)f_{R,IR}(xl_0) = f_{R_S, I}(x) = 
\left(\rank_SR\right)f_{S,I}(x),
\quad\mbox{for all}\quad 
x\in \R_{\geq 0},$$
where $l_0 = m_0/n_0$ is an integer.
Hence $f_{R, IR} \equiv f_{R_S, I}$ if $m_0 = n_0$.
\end{lemma}
\begin{proof} We first prove that $n_0$ divides $m_0$. Otherwise 
$m_0 = n_0l_0+n_1$, where $0<n_1<n_0$. Now if $x \in Q(S)$ is an homogeneous
element of degree $m_0$ and $y\in Q(R)$ is an homogeneous  element of degree 
$n_0$ then $(x)(y^{-l_0})$ is an homogeneous element of degree $n_1$ in 
$Q(R)$. By Lemma~\ref{r2}, this contradicts the hypothesis that  
$\mbox{gcd} \{n>0\mid R_n\neq 0\} =  n_0$.

Now, replacing  $R$ by $R^{(n_0)}$ and $S$ by $S^{(n_0)}$ we can 
assume $n_0 = 1$ and $m_0 = l_0$.
  
By definition 
$$f_{R_S, I}(x) = \lim_{n\to \infty}  
  \frac{1}{q^{d-1}}\left(\ell(R/I^{[q]}R)_{\lfloor xq\rfloor l_0}+
\cdots +\ell(R/I^{[q]}R)_{\lfloor xq\rfloor l_0+ l_0-1}\right).$$
and  $$f_{R, IR}(xl_0) = 
\lim_{n\to \infty}\frac{1}{q^{d-1}}
\left(\ell(R/I^{[q]}R)_{\lfloor xl_0q\rfloor}\right).$$

For all $x\geq 0$ and $q=p^n$, we have 
$|\lfloor xql_0\rfloor - \lfloor xq\rfloor l_0 |\leq l_0$. Let $m_1$ 
be such that $R_m\neq 0$, for $m\geq m_1$.
 Then for each  $0\leq l_i \leq 2l_0$, we have 
generically isomorphic graded maps 
$R(-l_i) \longto R(m_1)$ and $R \longto R(m_1)$  of degree $0$. Now by
Lemma~\ref{l12}, 
there is a constant $C_{l_0}$ such that 
$$\frac{1}{q^{d-1}}|\ell(R/I^{[q]})_{\lfloor xql_0\rfloor}-
\ell(R/I^{[q]})_{\lfloor xql_0\rfloor+l_i}|\leq C_{l_0},
\quad\mbox{for all}\quad x,~q~\mbox{and}~~0\leq l_i\leq 2l_0$$
which implies 
$$f_{R, I}(xl_0) = \lim_{n\to \infty}\frac{1}{q^{d-1}}\ell(R/I^{[q]})_{\lfloor 
xql_0\rfloor} = 
\lim_{n\to \infty}\frac{1}{q^{d-1}}\ell(R/I^{[q]})_{\lfloor xq\rfloor l_0 + 
l_i}.$$ 
This proves the lemma.
\end{proof}

\section{Proof of the Main Lemma}
Here we prove that the Main technical  Lemma~\ref{ML} which  will complete 
the proof 
of the Main Theorem.

Throughout this section we follow the Notations~\ref{n5}.

As we mentioned earlier, the  sequence~(\ref{e10}) need 
not be locally split exact as the sheaf
 $\sO_{m+q}$ is not invertible in general. 
Hence we consider the following locally split 
(as 
$\sO_{{\lfloor m+q/r\rfloor}r} \simeq \sL^{{\lfloor m+q/r\rfloor}}$ is 
invertible)
 exact sequence of 
$\sO_X$-modules 
 
\begin{equation}\label{e20}
0\longto \sG_{m,q}\longto  \oplus_{i=1}^s \sO_{(m_{q, d_i})}
\longby{\bar {\varphi}_{m, q}}
\sL^{{\lfloor m+q/r\rfloor}}\longto 0,\end{equation}
where
${\bar \varphi}_{m,q}(a_1,\ldots, a_s) = \sum_ia_if_i^q$
and where $\sG_{m,q} = \sF_{{\lfloor m+q/r\rfloor}r-q, q}$ (note however that 
$\sG_{m,q}$ may not be locally free).
In case 
$m+q$ is divisible by $r$, the sequence~(\ref{e20}) is same as the sequence 
$$0\longto \sF_{m,q}\longto  \oplus_{i=1}^s \sO_{m_{q, d_i}}
\longby{ {\varphi}_{m, q}}
\sO_{m+q}\longto 0,$$
as in (\ref{e10}).

In Lemma~\ref{l2}, we show  
that the length of the cohomology of 
$\sF_{m,q}$ (or of $\sO_{m+q}$) differs from the length of the  
cohomology of $\sG_{m,q}$ 
($\sL^{\lfloor m+q/r\rfloor}$ respectively)  by a function of order $O(m+q)^{d-2}$, where 
$\dim~X = d-1$. Hence it will be sufficient  to prove  
the Main Lemma~\ref{ML} for $\sG_{m,q}$ instead of $\sF_{m,q}$.

In the rest of the section we use the following 

\vspace{10pt}

\noindent{\bf Terminology}\quad We fix
a pair 
$(R, I)$, the line bundle ${\sL}$, the integer $r$ along with 
a choice of generators $f_1, \ldots, f_s$ of $I$ as in Notations~\ref{n5}.

Given $L$, where $L$ might be  a number, a set, a map  or 
a coherent sheaf, $C_L$ denotes a constant which depends only on $L$
(with the fixed data $(R, I)$, $\sL$ etc. as above).  
By supp dim $\sF$, we mean the dimension of the support of $\sF$.

Here  supp dim $\sO_n = $ supp dim $X = d-1\geq 1$.

\begin{lemma}\label{l5}For a given coherent sheaf $\sN$ of $\sO_X$-modules
 with  {\em supp dim} $\sN <d-1$, there is a constant $C_{\sN}$ such that
\begin{enumerate}
\item $h^j(X, \sN\tensor \sL^m) \leq C_{\sN}(|m|)^{d-2}$, for every 
$j\geq 0$ and $m\in \Z$.
\item $h^j(X, \sO_m \tensor \sN)) \leq  C_{\sN}(|m|)^{d-2}$, for all 
 $j\geq 0$ and for all $m\in \Z$. 

In particular, for $m\geq 0$ and $0\leq j \leq d-1$,
\item there exists $C$ such that $h^j(X, \sO_{m_q}\tensor \sN) \leq C_{\sN}(m+q)^{d-2}$, 
where, for $1\leq i \leq s$, 
${m_q} = m_{q,d_i}$,  ${m_q} = (m_{q,d_i})$ or ${m_q} =m+q$ (as 
in Notations~\ref{n5}), and 
\item $h^j(X, \sG_{m,q}\tensor \sN)) \leq  C_{\sN}(m+q)^{d-2}$, for all 
$m, j\geq 0$ and $q$.\end{enumerate} \end{lemma} 

\begin{proof} (1)\quad By the Serre vanishing theorem ([H])
 $h^j(X, \sN\tensor \sL^m) = 0$, for $j>0$ and $m>>0$. 
Also, for $m>>0$,  $h^0(X, \sN\tensor \sL^m)$ 
is a polynomial  of degree equal to $\dim \sN < d-1$.
 Hence the assertion~(1) follows by induction on  $\dim \sN$ and 
the  Serre's duality ([H]).
\vspace{5pt}

\noindent{(2)}\quad By Lemma~\ref{r1}, we have $\sO_m = 
\sL^{{\lfloor m/r\rfloor}}\tensor \sO_{r_1}$, where 
$r_1 = m-{\lfloor m/r\rfloor}r <r$.
 Since the support of $\sO_m\tensor \sN = $ the support of $\sN$, the 
assertion (2) follow from the fact that  
$\sO_{r_1}\tensor \sN$ belongs to  the finite set 
$\{\sO_0\tensor \sN, \sO_1\tensor\sN, \ldots, \sO_{r-1}\tensor\sN\}$ of 
coherent sheaves of $\sO_X$-modules. 

The assertion~(3) follows from (2) as  $|m+q-d_iq| \leq d_i(m+q)$.

 Since the sequence~(\ref{e20}) is locally split exact, the induced sequence 
$$0\longto \sG_{m,q}\tensor \sN \longto \oplus_{i=1}^s \sO_{(m_{q, d_i})}
\tensor \sN \longby{{\bar {\varphi}_{m, q}}\tensor \sN} \sL^{{\lfloor m+q/r\rfloor}}\tensor \sN\longto 0 $$
is exact. Now the assertion~(4) follows from (2) and (3).
\end{proof}

\begin{lemma}\label{l8} (1)\quad  
Let 
$S = \{\psi_{j}:\sE_{j}\longto \sF_{j}\mid 1\leq j\leq s\}$ be a finite  
set of  $\sO_X$-linear maps, where 
 $\sE_j$ and $\sF_{j}$ are  
coherent sheaves of $\sO_X$-modules.
For  $m\in \Z$, let  
$${\psi}_j(m) := Id_{\sL^{m}}\tensor \psi_{j}: \sL^{m}\tensor \sE_{j}
\longto  \sL^{m}\tensor \sF_{j}$$
 be the canonically induced maps. Assume that 
{\em supp dim} $(\ker \psi_{j})$ and
supp dim $(\coker \psi_{j})$ are each $< {d-1}$. 
Then there exists a constant $C_S$ such that 
$$h^i(X, \ker {\psi}_{j}(m)) \leq C_S m^{d-2}\quad\mbox{and}\quad
h^i(X, \coker {\psi}^{m}_{j}) \leq C_S{m}^{d-2},\quad
\mbox{for all}\quad i\geq 0.$$

\noindent{(2)}\quad
Moreover if\quad  
$\{0\longto \sN'_{m}\longto \sM'_{m}\longby{\phi_m} 
{\sM}_{m}\longto {\sN}_{m}\longto 0\}_{m\in \Z}$
denote a family of exact sequences of $\sO_X$-modules and $C_1$ and 
$C_2$ are constants such that  
$$h^i(X, \sN'_{m}) \leq C_1({n_m})^{d-2}\quad\mbox{and}\quad 
h^i(X, {\sN}_{m}) \leq 
C_2({n_m})^{d-2},\quad\mbox{for all}\quad i\geq 0,$$
then 
$$|h^0(X, {\sM'}_{m}) - h^0(X, {\sM}_{m}) |\leq (C_1+C_2)({n_m})^{d-2}.$$
\end{lemma}
\begin{proof}We note that, for any $m\in\Z$, 
$$\ker {\psi}_j(m)\simeq 
\sL^m\tensor \ker{\psi}_{j}\quad\mbox{and}\quad 
\coker {\psi}_j(m)\simeq 
\sL^m\tensor \coker {\psi}_j,$$
where $\ker \psi_j$ and $\coker \psi_j$ are  
in a fixed family of finite number of coherent sheaves of $\sO_X$-modules. 
Hence the first assertion follows by Lemma~\ref{l5}.

The second assertion follows by splitting the exact sequence into 
two  canonical two short exact sequences

\begin{equation*}\begin{aligned}0\longto \sN'_m\longto \sM'_{m}\longto 
\mbox{Im}(\phi_m)\longto 0,\\
0 \longto \mbox{Im}(\phi_m)\longto {\sM}_{m}\longto {\sN}_{m}\longto 0.
\end{aligned}\end{equation*}
\end{proof}

\begin{lemma}\label{l2}For all $m$ and  $q = p^n$,
\begin{enumerate} 
\item there is a constant $C$ such that
$$\left|h^0(X, \sG_{m,q}) - h^0(X, \sF_{m, q}) \right| \leq 
C(m+q)^{d-2}.$$

\item For given integer $l_0$, there exists a constant $C_{l_0}$
 such that for every $0\leq l\leq l_0$, 
$$\begin{array}{ccl}
\left|h^0(X, \sG_{m,q}) - h^0(X, \sG_{m+l, q}) \right|
& \leq & C_{l_0}(m+q)^{d-2},\\\\
\left|h^0(X, \sO_{m_q}) - h^0(X, \sO_{{m_q}+l}) \right|
& \leq & C_{l_0}(m+q)^{d-2},
\end{array}$$
where, for $1\leq j\leq s$, ${m_q} = m_{q,d_j}$ or ${m_q} = (m_{q,d_j})$,  
 or $ {m_q} = m+q$.
\end{enumerate}
\end{lemma}
\begin{proof}{\bf Claim}~(A).\quad For a given ${\tilde r}\in \Z$, if 
 $H^0(X, \sO_{\tilde r})\neq  \{0\}$ then  
there exists a constant $C_{\tilde r}$ such that, for $i\geq 0$ and $m\geq 0$
$$|h^0(X, \sF_{m,q})- h^0(X, \sF_{m+{\tilde r},q})|
\leq C_{\tilde r} (m+q)^{d-2}.$$

\noindent{\underline{Proof of the claim}}:\quad
An element $h\in H^0(X, \sO_{\tilde r})\setminus \{0\}$ gives an 
injective map
$\Phi_h:\sO_m\longto \sO_{m+{\tilde r}}$, for all $m$.
In particular we have the following canonical diagram of sheaves of $\sO_X$-modules:

$$\begin{array}{ccccccccc}

0 & \longto &  \sF_{m+{\tilde r},q} & \longto &  \oplus_{i=1}^s 
\sO_{m+q+{\tilde r}-qd_i} &
\longby{ {\varphi}_{m+{\tilde r}, q}} & 
\sO_{m+{\tilde r}+q} & \longto & 0\\

& & \uparrow^{\Phi'_h} & & \uparrow^{\oplus_i\Phi_h} & & \uparrow^{\Phi_h} & &{}\\

0 & \longto & \sF_{m,q} & \longto &  \oplus_{i=1}^s \sO_{m+q-qd_i}&
\longby{\varphi_{m, q}}& \sO_{m+q}&
\longto & 0\\
& & \uparrow  & & \uparrow & & \uparrow & &{}\\
& & 0 & & 0 & & 0 & &{}

\end{array}$$
The  map $\oplus_i\Phi_h:\oplus_i\sO_{m+q-qd_i}\longto 
\oplus_i\sO_{m+{\tilde r}+q-qd_i}$
is same as 
$$ \oplus_i (\mbox{Id}_{\sL}\tensor \phi_{i}): 
\oplus_i(\sL^{{\lfloor m_{q,d_i}/r\rfloor}}\tensor 
\sE_{i}) \longto  
\oplus_i(\sL^{{\lfloor m_{q,d_i}/r\rfloor}}\tensor 
\sF_{i} )$$
and  where 
$${\sE}_{i} = \sO_{{m+q-qd_i}-\lfloor m_{q,d_i}/r\rfloor r}\quad{and}\quad 
{\sF}_{i} = \sO_{{m+{\tilde r}+q-qd_i}-\lfloor m_{q,d_i}/r\rfloor r}$$
and the map $\phi_{i}:\sE_{i} \longto   
\sF_{i}$ is the multiplication map by  $h$.

Also the map $\Phi_h:\sO_{m+q}\longto \sO_{m+{\tilde r}+q}$ is 
$$ \mbox{Id}_{\sL}\tensor \phi_{0}: 
\sL^{{\lfloor m+q/r\rfloor}}\tensor 
\sE_{0} \longto  
\sL^{{\lfloor m+q/r\rfloor}}\tensor 
\sF_{0},$$
 where 
$${\sE}_{0} = \sO_{{m+q}-\lfloor m+q/r\rfloor r}\quad\mbox{and}\quad
 {\sF}_{0} = \sO_{{m+{\tilde r}+q}-\lfloor m+q/r\rfloor r}$$
and the map $\phi_{0}:\sE_{0} \longto   
\sF_{0}$ is given by the multiplication by $h$.
Note that  $\sE_{i}\in \{\sO_0,\ldots, \sO_{r-1}\}$ and 
$\sF_{i}\in \{\sO_{\tilde r},
\ldots, \sO_{{\tilde r}+r-1}\}$ and supp dim $(\coker~\phi_{i}) < d-1$. Hence 
the claim follows by  Lemmas~\ref{l5} and \ref{l8} and  the short exact sequence
$$0\longto \coker~\Phi'_h \longto \coker~(\oplus_i^s\Phi_h)\longto 
\coker~\Phi_h\longto 0.$$

\vspace{5pt}

\noindent{\underline{Assertion}~(2)}.
It is enough to prove the Assertion~(2) for $\sF_{m,q}$ instead of 
$\sG_{m,q}$.
Since there exists  $x_1\in H^0(X, \sO_{2r})\setminus \{0\}$, the 
above claim implies that 
we have 
a constant $C_{2r}$ such that
$$|h^0(X, \sF_{m,q})- h^0(X, \sF_{m+{2r},q})|
\leq C_{2r} (m+q)^{d-2},\quad\mbox{for}\quad i\geq 0.$$

\vspace{10pt}

\noindent{{\underline{Case}}}~1.\quad If $l\leq r$ then there exists 
 $x_2\in H^0(X, \sO_{2r-l})
\setminus \{0\}$, and therefore we have 
a constant $C_{2r-l}$ such that, for $i\geq 0$,
$$|h^0(X, \sF_{m+l,q})- h^0(X, \sF_{m+2r,q})|
\leq C_{2r-l} (m+q)^{d-2}.$$

\noindent{{\underline{Case}}}~2.\quad If $l\geq r$ then we can choose 
$x_3\in H^0(X, \sO_{l})\setminus \{0\}$ and 
therefore get  
a constant $C_{l}$ such that
$$|h^0(X, \sF_{m,q})- h^0(X, \sF_{m+l,q})|
\leq C_{l} (m+q)^{d-2},$$ for $i\geq 0$.
Since, for given $0\leq l\leq l_0$, there are finitely many choices of 
such $C_{l}$, we get Assertion~(2) of the lemma.
 
Similarly we prove the lemma  for  $\sO_{m_q}$. 

\vspace{5pt}

\noindent{\underline{Assertion}~(1)}.\quad It follows from the proof of 
Assertion~(2).\end{proof}

\subsection{The Main Lemma for $\sG_{m,q}$} 
Here we compare $h^0(X, \sG_{mp, qp})$ ($h^0(X, (\sO_{m_qp})$) 
with  $h^0(X, \oplus^{p^{d-1}}\sG_{m,q})$ 
($h^0(X, \oplus^{p^{d-1}}\sO_{m_q})$ respectively)
 in Lemma~\ref{l3} and in Lemma~\ref{l7}. Since the sequence~(\ref{e20}) is 
locally split exact, it remains exact for the functor $(-)\tensor \sM$, 
 for any sheaf of $\sO_X$-modules $\sM$.
In particular we construct below a generically isomorphic map 
$F^*\sG_{m,q}\longto \sG_{m', qp}$, provided  $|mp-m'|$ bounded 
by constant for all $m$ and $m'$.

\begin{lemma}\label{l3}
There is a constant $C_0$ such that
$$\begin{array}{ccc}
\left|h^0(X, (F^*\sG_{m,q})) - h^0(X, \sG_{mp, qp}) \right| 
& \leq  & C_0(mp+qp)^{d-2},\\\\
\left|h^0(X, (F^*\sO_{m_q})) - h^0(X, \sO_{{m_q}p}) \right| 
 & \leq  & C_0(mp+qp)^{d-2},
\end{array} $$
where ${m_q} = m_{q,d_j}$,  ${m_q} = (m_{q,d_j})$, for $1\leq j\leq s$,  or ${m_q} =m+q$ and where
$m\geq 0$ and $q = p^n$.
\end{lemma}
\begin{proof}{\bf Claim}.\quad For given $n$ there is a generically isomorphic map
 $\psi_n:F^*\sO_n \longto \sO_{np}$.

\vspace{5pt}

\noindent{\underline{Proof of the claim}}:\quad
By notation $\sO_n = \sO_X(nD)$, where $D$ is a $\Q$-Weil divisor. 
Let $D = \sum a_iD_i$, where $a_i\in \Q$ and $D_i$ are prime divisors.
Then 
$$\lfloor npD\rfloor = \sum_i\lfloor a_in\rfloor pD + \sum_i m_iD_i = 
p\lfloor nD\rfloor + \sum_i m_iD_i,$$
where $0\leq m_i\leq p$ are integers.
Let $\sM = \sO_X(p\lfloor nD\rfloor)$ then $\sM \longby{\bar f_n} \sO_{np}$ 
is an inclusion such that  supp dim $\coker {\bar f_n} < d-1$.

On the other hand, we can define the map 
${\phi_n}:F^*{\sO_n} \longto \sM$ as follows:
For the Frobenius map  $F:X_1\longto X$ let $F^*\sO_n = F^{-1}\sO_n
\tensor_{F^{-1}\sO_X}\sO_{X_1}$ and  
  let $\{D_+(f)\}_f$ denote
the affine open cover  of $X$, where 
 $f\in R$ is an homogeneous element of $R$. 
Then
the map $\phi_n\mid_{D_+(f)}$ is given by 
$$v/f^j\tensor u/f^i \to (v/f^j)^p\cdot u/f^i, \quad\mbox{if}\quad 
v/f^j \in F^{-1}\sO_n\quad\mbox{and}\quad u/f^i\in \sO_{X_1}.$$

The map $\phi_n$ is isomorphism on the regular locus $X_{reg}$ of $X$ as 
 $\sO_n\mid_{X_{reg}}$ is invertible. 
In particular $\psi_n = {\bar f_n}\cdot \phi_n$ is generically an isomorphism.
This proves the claim.

Now  the map $\psi = \oplus_i\psi_{(m_{q, d_i})}:\oplus_iF^*\sO_{(m_{q, d_i})}\longto 
\oplus_i\sO_{(m_{q, d_i})p}$ is generically an isomorphism, and 
$\phi$ is an isomorphism such that $\phi\circ F^*\phi_{m,q} = 
{\bar \phi}_{m', qp}\circ \phi$. This gives us  a map 
$F^*\sG_{m,q}\longto \sG_{m', qp}$ such that 
 the following diagram commutes 

$$\begin{array}{ccccccccc}

0 & \longto &  \sG_{m',qp} & \longto &  \oplus_{i=1}^s \sO_{(m_{q, d_i})p} &
\longby{\bar {\varphi}_{m', qp}} & 
\sL^{{\lfloor m+q/r\rfloor}p} & \longto & 0\\

& & \uparrow^{f_{m,q}} & & \uparrow^{\psi} & & \uparrow^{\phi} & &{}\\

0 & \longto & F^*\sG_{m,q} & \longto &  \oplus_{i=1}^s F^*\sO_{(m_{q, d_i})}&
\longby{F^*\varphi_{m, q}}& F^*\sL^{{\lfloor m+q/r\rfloor}}&
\longto & 0,
\end{array}$$

\vspace{10pt}

where  
$m'= \lfloor (m+q)/r\rfloor rp-qp$.  Therefore $m' = mp-r_1p$, for some 
$0\leq r_1<r$.
Note that 
the map $\psi_{(m_{q, d_i})}: F^*\sO_{(m_q,d_i)} \longto 
\sO_{(m_q,d_i)p}$
 is the same as the map
$$\psi_{{\lfloor m_{q,d_i}/r\rfloor}r}\tensor 
\psi_{{i_{j}}}:
F^*\sL^{{\lfloor m_{q,d_i}/r\rfloor}}\tensor F^*\sO_{i_j} 
\longto \sL^{{\lfloor m_{q,d_i}/r\rfloor}p}\tensor \sO_{{i_j}p},$$
where   the map
$\psi_{{\lfloor m_{q,d_i}/r\rfloor}r}:
F^*\sL^{{\lfloor m_{q,d_i}/r\rfloor}}\longto 
\sL^{{\lfloor m_{q,d_i}/r\rfloor}p}$ is an isomorphism
and the generically isomorphic map
$\psi_{i_j} \in  \{\psi_j:F^*\sO_j\longto \sO_{jp} \mid 
-r\leq j\leq r\}$.

Similarly for any $m\in \Z$, the map 
$\psi_m:F^*\sO_m\longto \sO_{mp}$ is same as the map
$$\psi_{\lfloor m/r\rfloor}\tensor \psi_{i}:
F^*\sL^{\lfloor m/r\rfloor}\tensor F^*\sO_i \longto 
\sL^{\lfloor m/r\rfloor p}\tensor \sO(ip),\quad\mbox{where}\quad 0\leq 
i<r$$
and where $\psi_{\lfloor m/r\rfloor}$ is an isomorphism.
Since the map $\phi$ is an isomorphism, we have $\ker~f_{m,q} = \ker~\psi$ and
$\coker~f_{m,q} = \coker~\psi$ and each have supp dim 
$ < d-1$. Hence
 the lemma follows by Lemma~\ref{l8}.
\end{proof}

\begin{lemma}\label{l7}There is a constant $C_1$ such that
$$\begin{array}{ccc}|p^{d-1}h^0(X,\sG_{m,q}) - 
h^0(X, F^*\sG_{m,q})| & \leq  & C_1(mp+qp)^{d-2}\\\\ 
|p^{d-1}h^0(X, \sO_{m_q}) - 
h^0(X, F^*\sO_{m_q})| &\leq  & C_1(mp+qp)^{d-2}, 
\end{array}$$ 
where ${m_q} = m_{q,d_j}$,  ${m_q} = (m_{q,d_j})$, for $1\leq j\leq s$,  or ${m_q} = m+q$, and where $m\geq 0$ and 
$q=p^n$.
\end{lemma}
\begin{proof} Recall $X= \mbox{Proj}~R = \mbox{Proj}~R^{r}$, where $R^{r}$
is  a standard graded domain. Therefore 
by Lemma~2.9 in [T], there is an integer 
$m_2\in \N$ (it will be a multiple of $r$) such that we have 
 a short exact sequence of sheaves of $\sO$-modules 
\begin{equation}\label{*}0\longto \oplus^{p^{d-1}}\sO_X({-m_2}D)\longby{\eta} 
F_*\sO_X\longto Q''\longto 0,\end{equation}
where support dimension $Q''$ is $<d-1$.

\vspace{10pt}
Let $M_1 = \oplus^{p^{d-1}}\sO_{-m_2}$ and $M = F_*\sO_X$.
Then  the short exact sequences 
$0\longto M_1\longby{\eta} M \longto Q''\longto 0$ and (\ref{e20}))
give the  following commutative diagram of canonical maps 
 
$$\begin{array}{ccccccccc}
& & 0 & & 0 & & 0 & &{}\\

& & \uparrow & & \uparrow & & \uparrow & &{}\\

0 & \longto &  \sG_{m,q}\tensor Q'' & \longto &  
\displaystyle{\oplus_{i=1}^s \sO_{(m_{q,d_i})}}\tensor Q'' & \longto  & 
\sL^{{\lfloor m+q/r\rfloor}}\tensor Q'' & \longto & 0\\

& & \uparrow^{h_{\sG_{m,q}}} & & \uparrow^{h_{\sL_{m,q}}} & & \uparrow & &{}\\

0 & \longto &  \sG_{m,q}\tensor M & \longto &  
\displaystyle{\oplus_{i=1}^s \sO_{(m_{q,d_i})}}\tensor M &
\longby{\bar {\varphi}_{m, q}} & 
\sL^{{\lfloor m+q/r\rfloor}}\tensor M & \longto & 0\\

& & \uparrow^{f_{\sG_{m,q}}} & & \uparrow^{f_{\sL_{m,q}}} & & \uparrow & &{}\\

0 & \longto & \sG_{m,q}\tensor M_1 & \longto & 
\displaystyle{\oplus_{i=1}^s \sO_{(m_{q,d_i})}}\tensor M_1 &
\longby{\varphi_{m, q}}& \sL^{{\lfloor m+q/r\rfloor}}\tensor M_1&
\longto & 0\\
& & \uparrow & & \uparrow & & \uparrow & &{}\\
& & \ker(f_{\sG_{m,q}}) & =   &  \ker(f_{\sL_{m,q}})& & 0, & &\\
& & \uparrow & &    \uparrow & & & &\\
& &  0 & & 0  & &  & &{}
\end{array}$$

Since $\sO_{(m_{q,d_i})} = \sL^{{\lfloor m_{q,d_i}/r\rfloor}}\tensor 
\sE_{i}$,  
where $\sE_{i} = \sO_{(m_{q, d_i})-{\lfloor m_{q,d_i}/r\rfloor}r} 
\in \{\sO_0, \sO_1, \ldots, \sO_{2r}\}$
the map $f_{\sL_{m,q}}$  is same as the map
$$\oplus_{i=1}^s \mbox{Id}_{\sL}\tensor \psi_{i}:
\oplus_{i=1}^s\sL^{{\lfloor m_{q,d_i}/r\rfloor}}\tensor \sE_{i}\tensor M_1 
\longto \oplus_{i=1}^s\sL^{{\lfloor m_{q,d_i}/r\rfloor}}\tensor 
\sE_{i} \tensor M,$$
where the map $\psi_{i} = Id_{\sE_{i}}\tensor \eta: \sE_{i} 
\tensor M_1\longto \sE_{i}\tensor M$  is generically isomorphic.
Therefore 
$$\mbox{supp dim of} \left(\ker(f_{\sG_{m,q}}) = \ker(f_{\sL_{m,q}}) = 
\oplus_{i=1}^s(\sL^{{\lfloor m_{q,d_i}/r\rfloor}}\tensor 
\ker~\psi_{i})\right) < d-1.$$
Now the long exact sequence 
$$0\longto \ker(f_{\sG_{m,q}}) \longto \sG_{m,q}\tensor M_1
\longto \sG_{m,q}\tensor M\longto \sG_{m,q}\tensor Q''\longto 0$$
gives (Lemma~\ref{l8}~(2))
\begin{equation}\label{e2}|h^0(X, \sG_{m,q}\tensor M_1) - h^0(\sG_{m,q}\tensor
F_*\sO_X)| = C_{\eta}(m+q)^{d-2},\end{equation}
 for some constant
$C_{\eta}$. 
On the other hand, as $F$ is a finite map, for any coherent sheaf $M$ of 
$\sO_X$-modules
the projection formula 
$F_*(F^*\sG_{m,q}\tensor M) = \sG_{m,q}\tensor F_*M$ holds.

This implies 
\begin{equation}\label{e3}
h^0(X, \sG_{m,q}\tensor F_*\sO_X) = h^0(X, F_*(F^*\sG_{m,q})) = 
h^i(X, F^*\sG_{m,q})\end{equation}
Now the lemma follows by (\ref{e2}) and (\ref{e3}).

The second assertion follows by the same line of arguments.
\end{proof}

\vspace{10pt}

\noindent{\underline {Proof of Main Lemma}~\ref{ML}}\quad 
It follows from  Lemma~\ref{l3}, Lemma~\ref{l7} and Lemma~\ref{l2}~(1).
$\Box$

\section{HK density functions for segre products of graded rings}
Here we show that the HK density function is multiplicative.
Let  $(R, I)$ and $(S, J)$ be two pairs, where  
$R = \oplus_{n\geq 0}R_n$ and $S =\oplus_{n\geq 0} S_n$ are graded domains  
of dimension $d_1\geq 2$ and $d_2\geq 2$
respectively, over a perfect field $k$, and 
$I\subset R$ and $J\subset S$ are  graded ideals of finite colengths.

Moreover let $F_R:[0, \infty)\longto [0,\infty)$ and 
$F_S:[0, \infty)\longto [0, \infty)$ be the Hilbert-Samuel density functions 
 given by
$$ F_R(x) = e_0(R)x^{d_1-1}/{(d_1-1)!},\quad\mbox{and}\quad  F_S(x) = 
e_0(S)x^{d_2-1}/{(d_2-1)!}, $$
where, for a graded ring $R$,  $e_0(R)$ is 
the Hilbert-Samuel multiplicity of  $R$ with respect to its graded 
 maximal ideal.

In [T], we had proved that the HK density function is multiplicative for 
Segre products of standard graded rings. In Theorem~\ref{t2} and Corollary~\ref{c2}, we show that this property extends to graded domains.

\begin{thm}\label{t2}For the pairs $(R, I)$ and $(S,J)$ as above if
$\mbox{gcd}~\{m\mid R_m \neq 0\} = 1$ and $\mbox{gcd}~\{m\mid S_m \neq 0\} = 1$. Then 
the HK density function of the pair $(R\#S, I\#J)$, where 
  $R\#S = \oplus_{n\geq 0}R_n\tensor_kS_n$ is the  Segre product of $R$ and $S$,
is given by 
$$F_{R\# S}- f_{R\#S, I\# J} = \left[F_{R}- f_{R,I}\right]
\left[F_{S}- f_{S,J}\right]$$
and also 
$$\begin{array}{lll}
e_{HK}(R\# S, I\# J) & = & \frac{e_0(R)}{(d_1-1)!}\int_0^{\infty} x^{d_1-1}
f_{S,J}(x) dx  +
\frac{e_0(S)}{(d_2-1)!}\int_0^{\infty} x^{d_2-1}
f_{R,I}(x) dx\\\\
& &   - \int_0^{\infty} f_{R,I}(x) f_{S,J}(x)dx.
\end{array}$$    
\end{thm}
\begin{proof}
Note that $R\#S$ is a graded integral domain with  
gcd $\{n\mid (R\# S)_n\neq 0\} = 1$.
Therefore 
$$(q^{d_1-1}q^{d_2-1})f_n(R\# S, I\# J)(m/q) = 
\ell({R\# S}/{(I\# J)^{[q]}})_{m}$$

$$ = \ell(R_{m})\ell(S_{m}) - 
\left[\ell(R_{m})-\ell({R}/{I^{[q]}})_{m}\right]
\left[\ell(S_{m})-\ell({S}/{J^{[q]}})_{m}\right].
$$
Hence 
$$f_n(R\# S, I\# J)(x)  = 
  f_n(S,J)\frac{\ell(R)_{\lfloor xq\rfloor }}{q^{d_1-1}}
+f_n(R, I)\frac{\ell(S)_{\lfloor xq\rfloor }}{q^{d_2-1}} - f_n(R,I)f_n(S,J).$$

Since $\{f_n(R\#S, I\#J)\}_n$, $\{f_n(R, I)\}_n$ and
$\{f_n(S, J)\}_n$ are uniformly convergent sequences with bounded supports,
taking limit as $n\to \infty$ we get, 
$$f_{R\# S, I\# S}(x) = F_R(x)f_{S,J}(x)+F_S(x)
f_{R,I}(x)-
f_{R,I}(x)f_{S,J}(x)\quad\mbox{for all}\quad x\geq 0.$$
 The rest of the proof follows as  $F_{R\# S}(x) = F_R(x)F_S(x)$.
\end{proof}

\begin{notations}
For a graded domain $R = \oplus_{n\geq 0}R_n$  with a graded 
ideal $I = \oplus_{n\geq 1}I_n$, we denote  
$R^{(m)} = \oplus_n R_{nm}~~\mbox{with degree}~~n~~\mbox{component}~~ = R_{nm}$

and
${I}^{(m)} = {I}\cap R^{(m)} = \oplus_n {I}_{nm}.$
\end{notations}

\begin{cor}\label{c2}If
$\mbox{gcd}~\{m\mid R_m \neq 0\} = n_1$ and $\mbox{gcd}~\{m\mid S_m \neq 0\} = n_2$.
Then 
$$F_{R\# S}- f_{R\#S, I\# J} = \left[F_{R^{(l)}}- f_{R^{(l)},I^{(l)}}\right]
\left[F_{S^{(l)}}- f_{S^{(l)},J^{(l)}}\right]$$
and 
$$\begin{array}{lll}
e_{HK}(R\# S, I\# J) & = & \frac{e_0(R^{(l)})}{(d_1-1)!}\int_0^{\infty} x^{d_1-1}
f_{S^{(l)},J^{(l)}}(x) dx  +
\frac{e_0(S^{(l)})}{(d_2-1)!}\int_0^{\infty} x^{d_2-1}
f_{R^{(l)},I^{(l)}}(x) dx\\\\
& &   - \int_0^{\infty} f_{R^{(l)},I^{(l)}}(x) f_{S^{(l)},J^{(l)}}(x)dx,
\end{array}$$    
where $l= \mbox{lcm}(n_1, n_2)$.
\end{cor}
\begin{proof}Since $R\#S= R^{(l)}\#S^{(l)}$ and $I\#J = I^{(l)}\#J^{(l)}$ and
$\mbox{gcd}~\{m\mid R^{(l)}_m \neq 0\} = \mbox{gcd}~\{m\mid S^{(l)}_m \neq 0\} = 1$,
the corollary follows from the above theorem.\end{proof}

\section{Applications and examples}

The coefficients of the HK function for  a pair $(R, I)$ (given as $HK(q) = 
1/q^d\ell(R/I^{[q]})$) 
 have a nice geometric description (see [K]) provided 
$\mbox{ proj dim}_R(R/I) < \infty$.
Here we prove that the HK density function too has a nice description 
in the  case of such graded pairs.

\begin{propose}\label{p2}Let $(R, I)$ be  a graded pair such that 
 $\mbox{proj dim}_R(R/I) < \infty$ then 
the HK density function $f_{R,I}$ is a piecewise polynomial 
function of degree $d-1$, where $f_{R,I}$ (and hence $e_{HK}(R, I)$)
is  given 
in terms of the graded Betti numbers of the minimal graded $R$-resolution of $R/I$.
\end{propose}

\begin{proof}Consider  the minimal graded resolution of $R/I$ 
over the graded ring $R$
$$0\longto \displaystyle{\oplus_{j\in \Z}R(-j)^{\beta_{d,j}}}
\longto \oplus_{j\in \Z}R(-j)^{\beta_{d-1,j}}\longto \cdots
\longto \oplus_{j\in \Z}R(-j)^{\beta_{1,j}}\longto R \longto R/I\longto 0.$$

Since the functor of Frobenius is exact on the category of 
modules of finite type and finite projective dimension (a corollary 
of the acyclicity lemma by Peskine-Szpiro [PS]), 
 we have a long exact sequence
$$0\longto \oplus_{j\in \Z}R(-qj)^{\beta_{d,j}}
\longto \oplus_{j\in \Z}R(-qj)^{\beta_{d-1,j}}\longto \cdots
\longto \oplus_{j\in \Z}R(-qj)^{\beta_{1,j}}\longto R \longto R/I\longto 0.$$

Let
${\tilde e_0} = e_0(R, {\bf m})/(d-1)!$ and 
let $$\B(j)= \beta_{0j}-\beta_{1j}+\beta_{2j}+\cdots + (-1)^d\beta_{dj}.$$

Note that  $\beta_{00}=1$ and  $\beta_{0, j}=0$ for $j\neq 0$. 
For $j < 0$, $\B(j) = 0$.
If $l$ be the largest integer such that $\beta_{il}\neq 0$ for some $i$.
Then
$$\ell(R/I^{[q]})_m = \ell(R_m) + \B(1)\ell(R_{m-q})+
\B(2)\ell(R_{m-q}) + \cdots + {\B}(l)\ell(R_{m-lq})$$
and therefore

$$\begin{array}{rll}
f_{R,I}(x) = & {\tilde e_0}\left[x^{d-1}\right]  & 0 \leq x \leq 1\\
 = & {\tilde e_0}\left[x^{d-1}+\B(1)(x-1)^{d-1}\right]  & 1 \leq x \leq 2\\

 & \vdots &\\
 = & {\tilde e_0}\left[x^{d-1}+\B(1)(x-1)^{d-1} +\cdots + \B(i)(x-i)^{d-1}
\right]  & i \leq x \leq (i+1)\\

 = & {\tilde e_0}\left[x^{d-1}+\B(1)(x-1)^{d-1} +\cdots + \B(l-1)(x-l+1)^{d-1}
\right]  & l-1 \leq x \leq l\\
= & 0 & l\leq x.\end{array}$$

Note that  $f_{R,I}$ is a compactly supported function which implies that
the polynomial
$$ x^{d-1}+\B(1)(x-1)^{d-1} +\cdots + \B(l)(x-l)^{d-1} = 0.$$
Hence $\supp(f_{R, I})\subseteq [0, l]$.

Moreover $$e_{HK}(R,I) = \frac{e_0}{d!}\left[\B(0)l^d+\B(1)(l-1)^d +
\cdots + \B(i)(l-i)^d +\cdots +\B(l-1)\right]$$
\end{proof}

\vspace{10pt}

\noindent{\underline {Proof of Corollary~\ref{c3}}}:
Let $m_0 = \mbox{gcd} \{n>0\mid S_n \neq 0\}$ and 
$n_0 = \mbox{gcd} \{n>0\mid R_n \neq 0\}$ and $l_0 = m_0/n_0$.
 Then, by Theorem~\ref{t1} and Lemma~\ref{l13}, for $x\geq 0$,  
$$f_{S, I}(x) = \frac{1}{\rank_SR}f_{R_S, I}(x) = 
\frac{l_0}{\rank_SR}f_{R, IR}(xl_0)\quad\mbox{and}\quad 
e_{HK}(S, I) = \frac{e_{HK}(R,IR)}{\rank_SR}.$$
Hence the corollary follows from Proposition~\ref{p2}.\hspace{50pt}\hfill$\Box$

\vspace{10pt}

\begin{rmk}If $(R,I)$ is a graded pair such that
   $R/IR$ has the  finite pure resolution 
$$0\longto \oplus^{\beta_d}R(-j_d)\longto \cdots \longto \oplus^{\beta_2}R(-j_2)
\longto \oplus^{\beta_1}R(-j_1)\longto R\longto R/I\longto 0$$
then $j_1<j_2 <\cdots <j_d$
and
$$\B(1) = \cdots = \B(j_1-1) = 0 ~~and ~~\B(j_1) = -\beta_1$$
$$\B(j_{n-1}+1) = \cdots = \B(j_n-1) = 0 ~~and ~~\B(j_n) = (-1)^n\beta_n.$$
Hence

$$\begin{array}{rll}
f_{R,I}(x) = & {\tilde e_0}\left[x^{d-1}\right]  & 0 \leq x \leq j_1\\
 = & {\tilde e_0}\left[x^{d-1}-\beta_1(x-j_1)^{d-1}\right]  & j_1 \leq x 
\leq j_2\\

 = & {\tilde e_0}\left[x^{d-1}-\beta_1(x-j_1)^{d-1} +\cdots +
(-1)^{d-1} \beta_{{d-1}}(x-j_{d-1})^{d-1}
\right]  & j_{d-1} \leq x \leq j_d\\
= & 0 & j_d\leq x. \end{array}$$
Here the maximum support of $f_{R, I} = \alpha(R, I) = j_d$, as
$\beta_d \neq 0$.
\end{rmk}

\vspace{10pt}

\subsection{Some concrete examples}

We recall the Hilbert-Burch theorem (see [BH]).

\begin{thm}Let $\psi:R^n\longto R^{n+1}$ be a $R$-linear map, where $R$ 
is a Noetherian ring. Let $I = I_n(\psi)$ be the ideal generated by 
$n+1$  elements consisting of $n\times n$ 
minors of the matrix given by $\psi$.

Then $\mbox{grade}~I_n(\psi)\geq 2$ implies the ideal $I$ has the resolution
of the form 
$$0\longto R^n\longby{\psi} R^{n+1} \longto I \longto 0.$$
\end{thm}

\vspace{10pt}

In the following examples we compute the HK density function 
$f_{S, I}$, where $S= R^G $ is the ring of invariants in $R= k[x_1, x_2]$ 
with  $G\in \{A_n, D_n, E_6, E_7, E_8\}$
and $I\subset S$ is its graded maximal ideal. This also recovers the 
computations 
of $e_{H}(R^G, I)$ given in Theorem~5.1 of [WY].
For this it is enough to construct the minimal graded resolution of 
$IR$ as a $R$-module.

Note that in all the following cases $R^G = k[h_1, h_2, h_3]\subset k[x_1, x_2]$, where 
$h_1, h_2, h_3$ are explicit homogeneous polynomials 
in $x_1, x_2$ (see  Chap~X, page 225 of [MBD]). 
 Using the Hilbert-Burch theorem, we will construct a
$R$-resolution for $IR = (h_1, h_2, h_3)R$, 
which is of the folllowing type:
\begin{equation}\label{re1}0\longto R(-l_1)\oplus R(-l_2)\longby{\psi}
R(-\deg~h_1)\oplus R(-\deg~h_2)\oplus R(-\deg~h_3)\longby{\phi} IR\longto 0,
\end{equation} 
where $\phi$ is given by the matrix $[h_1, h_2, h_3]$. In the forthcoming 
 set of examples we define the map $\psi$ by giving a $3\times 2$ matrix in $R$
(this will also determine the  values $l_1, l_2$).
Since $\mbox{grade}~IR = 2$, to prove that (\ref{re1}) is exact, 
 it only remains to check that $I_2(\psi) = IR$ which can be done 
 easily.

\vspace{5pt}

\begin{ex}\label{A_n} Let $G = A_n$ then $|G| = n\geq 2$ and 
$\Char~k = p \geq 2$ and $(p, n) = 1$. 

$$R^G = k[h_1, h_2, h_3] \cong
\frac{k[x_1, x_2, x_3]}{(x_1^n+x_2x_3)},$$
 where $h_1 = x_1x_2,~~ h_2 = x_1^n$ and $h_3 = x_2^n$.
The map  $\psi$  is given by the matrix 
$$
\begin{bmatrix}x_1^{n-1} & -x_2 & 0\\
x_2^{n-1} & 0  & -x_1 \end{bmatrix} .$$
Then the sequence
$$0\longto R(-n-1)\oplus R(-n-1) \longby{\psi} R(-2)\oplus R(-n) 
\oplus R(-n) \longto IR \longby{\phi} 0$$
is the minimal resolution for $IR$
as $I_2(\psi) = I$.
Here $\B(2) = -1$, $\B(n) = -2$ and $\B(n+1) = 2$.
If $n$ is even then the HK density function $f_{S,I}$ is given by 
$$\begin{array}{lcll}
f_{S,I}(x) & = & 4x/(n+1) & \mbox{if}\quad 0\leq x \leq 1\\
& = & 4/(n+1) & \mbox{if}\quad 1\leq x\leq n/2\\
& = & 2(2-4x+2n)/(n+1) & \mbox{if}\quad n/2 \leq x\leq (n+1)/2
\end{array}
$$

If $n$ is odd then the HK density function $f_{S,I}$ is given by 
$$\begin{array}{lcll}
f_{S,I}(x) & = & x/(n+1) & \mbox{if}\quad 0\leq x \leq 2\\
& = & 2/(n+1) & \mbox{if}\quad 2\leq x\leq n\\
& = & (2-2x+2n)/(n+1) & \mbox{if}\quad n \leq x\leq n+1
\end{array}$$
\end{ex}

\vspace{5pt}

\begin{ex}\label{D_n}Let $G= D_n$ the dihedral group then $|G| = 4n$ and
$\Char~k = p\geq 3$ and $(p, n) = 1$.
$$R^G = k[h_1, h_2, h_3]
= \frac{k[x_1, x_2, x_3]}{(x_3^2+x_1 x_2^2+ x_1^{n+1})},$$    
where 
$$h_1 = -2x_1^2x_2^2,~~ h_2 = x_1^{2n}+(-1)^nx_2^{2n}, 
h_3 = x_1x_2(x_1^{2n}-(-1)^nx_2^{2n}.$$

We assume $n$ is even.
Let map $\psi$ is  given by the matrix 
$$
\begin{bmatrix}-2x_1^{n-1} & x_1x_2^2 & x_2\\
-2x_2^{n-1} & -x_1^2x_2  & x_1 \end{bmatrix}$$
Then the sequence
$$0\longto R(-2n-3)\oplus R(-2n-3) \longby{\psi} R(-4)\oplus R(-2n) 
\oplus R(-2n-2) \longby{\phi} IR \longto 0$$
is the minimal resolution for $IR$
as $I_2(\psi) = IR$.

Here $\B(4) = -1$, $\B(2n) = -1$ and $\B(2n+2) = -1$.
If $n$ is even then the HK density function  $f_{S,I}$ is given by 
$$\begin{array}{lcll}
f_{S,I}(x) & = & x/n-2 & \mbox{if}\quad 0\leq x \leq 2\\
& = & 2/(n-2) & \mbox{if}\quad 2\leq x\leq n\\
& = & (n+2-x)/(n-2) & \mbox{if}\quad n \leq x\leq n+1\\
& = & (2n+3-2x)/(n-2) & \mbox{if}\quad n+1\leq x\leq n+3/2
\end{array}
$$
and hence $e_{HK}(R^G, I) = 2-1/4n$.
\end{ex}

\vspace{5pt}

\begin{ex}\label{e6}
Let $G = E_6$ the tetrahedral group then $|G| = 24$ and $\Char~k = p\geq 5$. 
 $$R^G = k[h_1, h_2, h_3] = 
\frac{k[x_1,x_2,x_3]}{(6a x_1^2- x_2^3+x_3^3)},\quad\mbox{where}\quad 
a= 2\sqrt{-3}$$

$$h_1 = x_1^5x_2-x_1x_2^5,\quad h_2 = x_1^4+ax_1^2x_2^2 + x_2^{4}, \quad
h_2 = x_1^4 - ax_1^2x_2^2 + x_2^{4}.$$

Let $\psi$ be given by the matrix
$$
\begin{bmatrix}x_1 &  -(a/2)x_1^2x_2-x_2^3 &  (a/2)x_1^2x_2-x_2^3 \\
x_2 &  x_1^3+(a/2)x_1x_2^2   &  x_1^3 - (a/2)x_1x_2^2   \end{bmatrix}$$
Then $I_2(\psi) = (ah_1, h_2, h_3)R$.
If $a\neq 0$ in $k$ then the canonical sequence 
$$0\longto R(-7)\oplus R(-7) \longby{\psi} R(-6)\oplus R(-4) \oplus R(-4)
\longby{\phi} IR\longto 0$$
is the minimal $R$-resolution of $IR$.

Here $\B(4) = -2$, $\B(6) = -1$ and $\B(7) = 2$ and the HK 
density function $f_{S,I}$ is given by 
$$\begin{array}{lcll}
f_{S,I}(x) & = & x/6 & \mbox{if}\quad 0\leq x \leq 2\\
& = & (4-x)/6 & \mbox{if}\quad 2\leq x\leq 3\\
& = & (7-2x)/6 & \mbox{if}\quad 3 \leq x\leq 7/2\\
& = & 0 & \mbox{otherwise}
\end{array}
$$
\end{ex}

\vspace{5pt}

\begin{ex}\label{e7}
Let $G = E_7$ octahedral group then $|G|= 24$ and $\Char~k \geq 5$
 and $$R^G = k[h_1, h_2, h_3] = 
\frac{k[x_1, x_2, x_3]}{(108x_1^4-x_2^3+x_3^2)},\quad\mbox{where}$$ 
$$h_1 = x_1^5x_2-x_1x_2^5, \quad h_2 = x_1^8+14x_1^4x_2^4 + x_2^{8}, \quad
h_3 = x_1^{12}-33(x_1^{8}x_2^4)- 33(x_1^4x_2^8) + x_2^{12}.$$

Let $\psi$ be given by the matrix
$$
\begin{bmatrix}-7x_1^{4}x_2^{3} - x_2^7 & x_1^5  & x_1\\
7x_1^{3}x_2^{4} +  x_1^7 &  x_2^5  & x_2  \end{bmatrix}$$

If $\Char~k >3$ then
 $(h_1, h_2, h_3)R =  I_2(\psi)$ and hence  
the minimal $R$-resolution for $IR$ is given by 

$$0\longto R(-13)\oplus R(-13) \longby{\psi} R(-6)\oplus R(-8) 
\oplus R(-12) \longby{\phi} IR \longto 0.$$

Here $\B(6) = -1$, $\B(8) = -1$, $\B(12) = -1$ and $\B(13) = 2$. Hence  the HK 
density of $(S,I)$ is given by 
$$\begin{array}{lcll}
f_{S,I}(x) & = & x/48 & \mbox{if}\quad 0\leq x \leq 6\\
& = & 6/48 & \mbox{if}\quad 6\leq x\leq 8\\
& = & (14-x)/48 & \mbox{if}\quad 8 \leq x\leq 12\\
& = & (26-2x)/48 & \mbox{if}\quad 12 \leq x\leq 13\\
& = & 0 & \mbox{otherwise}.
\end{array}
$$
and hence $e_{HK}(R^G, I) = 2-(1/24)$.

\end{ex}

\vspace{5pt}

\begin{ex}\label{e8}Let $G = E_8$ the icosahedral group then 
$|G| = 120$ and $\Char~k \geq 7$.
Now
$$R^G = k[h_1, h_2, h_3] = 
\frac{k[x_1, x_2, x_3]}{(x_2^2+x_3^3-1728 x_1^5)},$$ where 

$$h_1 = x_1x_2(x_1^{10}+11x_1^5x_2^5-x_2^{10})$$

$$h_2 = x_1^{30}+x_2^{30}+522(x_1^{25}x_2^5- x_2^{25}x_1^5) 
- 10005(x_1^{20}x_2^{10}+x_1^{10}x_2^{20})$$

$$h_3 = -x_1^{20}-x_2^{20}+228(x_1^{15}x_2^5- x_2^{15}x_1^5) 
-494(x_1^{10}x_2^{10}).$$

Let $\psi$ be given by the matrix
$$ \begin{bmatrix}
x_1 & f_2 & f_3\\
x_2 & g_2 & g_3\end{bmatrix}$$
where $$f_2 = -x^{11}-(11/2)x_1^6x_2^5. \quad\quad 
f_3 = x_2^{19}+ax_1^{5}x_2^{14}+(b/2)x_1^{10}x_2^9$$

$$g_2 = -x_2^{11}+(11/2)x_1^5x_2^6\quad \quad 
g_3 = -x_1^{19}+ax_1^{14}x_2^{5}-(b/2)x_1^9x_2^{10}$$

and where $a=228$ and $b= 494$.

In particular $(h_1, h_2, h_3)R = I_2(\psi)$

Hence the minimal $R$-resolution for $IR$ is given by 

$$0\longto R(-31)\oplus R(-31) \longby{\psi} R(-12)\oplus R(-30) 
\oplus R(-20) \longto IR \longto 0.$$

$\B(12) = -1$, $\B(20) = -1$, $\B(30) = -1$ and $\B(31) = 2$.

Hence  the HK 
density of $(S,I)$ is given by 
$$\begin{array}{lcll}
f_{S,I}(x) & = & x/30 & \mbox{if}\quad 0\leq x \leq 6\\
& = & 6/30 & \mbox{if}\quad 6\leq x\leq 10\\
& = & (16-x)/30 & \mbox{if}\quad 10 \leq x\leq 15\\
& = & (31-2x)/30 & \mbox{if}\quad 15 \leq x\leq 31/2\\
& = & 0 & \mbox{otherwise}
\end{array}
$$
and hence $e_{HK}(R^G, I) = 2 -(1/120)$.

\end{ex}

\vspace{10pt}

\begin{rmk}If $(R,I)$ is  a two dimensional graded pair then its
HK density function $f_{R,I}$ is an explicit  piecewise linear polynomial 
with rational coefficients and rational break points (the proof follows from the same arguments as in [TW]): 

 Let $f_1, \ldots, f_s$ 
be a set of homogeneous generators of $I$ of degrees 
$d_1, \ldots, d_s$. Let ${\tilde S}$ denote the  normalization of 
$R^{(n_0)} = \oplus{n\geq 0} R_{nn_0}$, where 
$\mbox{gcd}\,\{m\mid R_m\neq 0\} = n_0$.
Let $X= \mbox{Proj}({\tilde S})$, Then for  the $\Q$-Weil divisor  $D$ (which is 
Cartier in this case) corresponding to  the normal ring ${\tilde S}$ 
(as in Theorem~\ref{D}) the sheaf $\sO_n = 
\sO_X(D)$ is invertible. Hence 
the sequence (\ref{e10}) is 
$$0\longto F^{n*}V\tensor \sO_m\longto \oplus_i\sO_{m+q-qd_i}\longby{\phi_{m,q}} 
\sO_{m+q}\longto 0,$$  
where 
$$0\longto V\longto \oplus_i\sO_{1-d_i}\longby{\phi} \sO_1\longto 0,$$  
where $\phi(x_1, \ldots, x_s) = \sum x_if_i$.
This gives
$$f_{R,I}(x) = f_{V,\sO_1}(x)-f_{\oplus_i\sO_{1-d_i}, \sO_1}(x),
\quad\mbox{for}\quad x\geq 0,$$ 
where, for a vector bundle $E$ on $X$ with strong HN data
$(\{a_1, \ldots, a_{l+1}\}, \{r_1, \ldots, r_{l+1}\})$ and $d = \deg\,\sO_1$, 
the function 
$f_{E, \sO_1}$ denotes the HK density function of $E$ with respect to 
$\sO_1$  and is given by 
$$\begin{array}{lcl}
 x < 1-a_1/d & \implies & f_{E,\sO_1}(x) =
-\left[\sum_{i=1}^{l+1}a_ir_i+d(x-1)r_i\right] \\
1-a_i/d  \leq x < 1-a_{i+1}/d  & \implies & f_{E,\sO_1}(x) =
- \left[\sum_{k={i+1}}^{l+1}a_kr_k+d(x-1)r_k\right].\end{array}$$
\end{rmk}

\end{document}